\documentclass[a4paper,reqno,12pt]{amsart}

\usepackage{amsmath,amsthm,amssymb,amsfonts,mathtools}
\usepackage{amscd,tikz-cd}
\usepackage[all]{xy} 
\usepackage{comment,todonotes}

\usepackage[capitalize]{cleveref}
\theoremstyle{plain} 
\newtheorem{thm}{Theorem}[section]
\newtheorem*{thm*}{Theorem}
\newtheorem{lem}[thm]{Lemma}
\newtheorem{prop}[thm]{Proposition}
\newtheorem{cor}[thm]{Corollary}

\newtheorem{athm}{Theorem}[section]

\theoremstyle{definition}
\newtheorem{dfn}[thm]{Definition}
\newtheorem{eg}[thm]{Example}

\newtheorem{rem}[thm]{Remark}
\newtheorem{conj}[thm]{Conjecture}
\newtheorem*{conj*}{Conjecture}

\newtheorem*{nota-conv}{Notation and Convention}
\newtheorem*{ack}{Acknowledgements}

\theoremstyle{remark}
\newtheorem*{pf}{Proof}

\numberwithin{equation}{section}

\newcommand{\ZZ}{\mathbb{Z}}
\newcommand{\QQ}{\mathbb{Q}}
\newcommand{\RR}{\mathbb{R}}
\newcommand{\CC}{\mathbb{C}}
\newcommand{\PP}{\mathbb{P}}

\newcommand{\HH}{\mathbb{H}}

\newcommand{\C}{\mathcal{C}}
\newcommand{\D}{\mathcal{D}}

\def\H{\mathcal{H}}

\def\L{\mathcal{L}}

\def\O{\mathcal{O}}
\def\P{\mathcal{P}}

\DeclareMathOperator{\Aut}{Aut}
\DeclareMathOperator{\MCG}{MCG}

\DeclareMathOperator{\Hom}{Hom}

\DeclareMathOperator{\Num}{Num}
\DeclareMathOperator{\NS}{NS}
\DeclareMathOperator{\Pic}{Pic}

\DeclareMathOperator{\rk}{rk}
\DeclareMathOperator{\ord}{ord}

\DeclareMathOperator{\id}{id}

\DeclareMathOperator{\im}{im}

\DeclareMathOperator{\scl}{scl}

\DeclareMathOperator{\Gal}{Gal}
\DeclareMathOperator{\MW}{MW}
\newcommand{\bPic}{\mathbf{Pic}}
\DeclareMathOperator{\Spec}{Spec}

\newcommand{\note}[1]{\textcolor{red}{#1}}

\begin{document}
\title[Gap theorems and achirality]
{Gap theorems and achirality for automorphisms of K3 surfaces and Enriques surfaces}

\author{Kohei Kikuta}
\address{Department of Mathematics, Graduate School of Science, Osaka University, Toyonaka Osaka, 560-0043, Japan.
\vspace{1mm}
\newline
\hspace*{4mm}
School of Mathematics, 
The University of Edinburgh, 
James Clerk Maxwell Building, 
Peter Guthrie Tait Road, Edinburgh, Scotland, EH9 3FD, UK.}
\email{kikuta@math.sci.osaka-u.ac.jp}

\author{Yuta Takada}
\address{Mathematical Sciences, The University of Tokyo, 
Tokyo 153-8914, Japan; JSPS Research Fellow.
\vspace{1mm}
\newline
\hspace*{4mm}
Department of Mathematics, 
Saarland University, Saarbr\"ucken 66123, Germany.}
\email{ytakada@ms.u-tokyo.ac.jp}

\author{Taiki Takatsu}
\address{Department of Mathematics,
Tokyo University of Science,
2641 Yamazaki,
Noda-shi,
Chiba Prefecture 278-8510,
Japan}
\email{taiki\_takatsu@rs.tus.ac.jp}

\begin{abstract}
We prove gap theorems for entropy norms on automorphism groups of K3 surfaces, Enriques surfaces, 
and irreducible holomorphic symplectic manifolds. 
We also study the achirality of automorphisms of K3 surfaces and Enriques surfaces in terms of genus-one fibrations.
\end{abstract}

\maketitle
\date{}

\section{Introduction}

\subsection{Gap theorems and achirality}

Commutator length (or commutator norm) on a group, as well as its stabilization,
stable commutator length (abbreviated as \emph{scl}), has been studied since the 1970s (\cite[Remark 2.3]{scl}).
Motivated by the Margulis thick-thin decomposition of a closed hyperbolic manifold, 
Calegari investigated the relationship between stable commutator length on its fundamental group and the lengths of closed geodesics, 
and proved a (spectral) gap theorem for $\scl$ on the fundamental group \cite{Calegari2008Length}.

A phenomenon that does not occur for fundamental groups of closed hyperbolic manifolds is the existence of non-chiral (i.e. achiral) elements.
Here, an element $a$ of a group $G$ is said to be \emph{achiral} 
if $a^m$ is conjugate to $a^{-m}$ for some $m\in\ZZ_{>0}$. 
Since the $\scl$ of an achiral element is zero, achiral elements must be treated separately in any generalization of the gap theorem.
More generally, Calegari--Fujiwara \cite{Cal-Fuj10} proved the gap theorem for $\scl$ for chiral loxodromic elements in hyperbolic groups and mapping class groups, 
and observed that the gap phenomenon for $\scl$ arises from a certain form of hyperbolicity of these groups, namely, from a suitable isometric action on a (Gromov) hyperbolic space.
In the case of the action of the fundamental group of a closed hyperbolic manifold on its universal cover $\HH$, 
the chirality of a loxodromic element admits a geometric interpretation in terms of its geodesic axis in $\HH$:
chirality excludes symmetries (that is, conjugacies) giving rise to an ``anti-aligned'' geodesic axis.
This geometric picture of chirality continues to work for hyperbolic groups and also for their generalizations, 
such as WPD actions and acylindrical actions on hyperbolic spaces.

\subsection{Main results}

A major new feature of this paper is that we apply the above ideas to automorphism groups of algebraic varieties.
For an algebraic variety $X$, let $\Aut(X)$ denote its automorphism group.
In this paper, we consider the cases where $X$ is a K3 surface, an Enriques surface, or an irreducible holomorphic symplectic manifold.
In this setting, it has been established by the first-named author and the third-named author that 
these automorphism groups act naturally on the associated hyperboloid as geometrically finite Kleinian groups \cite{Kikuta,Takatsu}, 
and hence are relatively hyperbolic.
This makes it possible to study automorphism groups by means of geometric group theory.

We now state the main results of this paper. 

\subsubsection{Gap theorem for entropy norm}

In the context of the study of conjugation-invariant norms on groups,
Brandenbursky--Marcinkowski introduced the notion of entropy norm on the (Hamiltonian) diffeomorphism group of a real surface, 
defined as the word norm associated with the set of diffeomorphisms of zero topological entropy \cite{Brandenbursky-Marcinkowski2019}.
In other words, this norm measures how many dynamically simple diffeomorphisms are needed in order to express a given diffeomorphism as their product.

For automorphisms of algebraic varieties over the complex numbers $\CC$,
the topological entropy is described, by the Gromov--Yomdin theorem, in terms of the spectral radius of the induced action on cohomology groups.
This theorem has greatly stimulated the dynamical study of algebraic varieties.
Moreover, even for algebraic varieties not necessarily defined over $\CC$, one can define entropy homologically (or numerically).

In this paper, first, in the same spirit as in the case of diffeomorphism groups, 
we introduce the entropy norm $e$ on $\Aut(X)$ by using entropy, and prove a gap theorem for its stabilization $\overline{e}$.

\begin{athm}[\cref{Gap_thm: ent_Aut}]\label{Gap_thm: ent_Aut-intro}
Let $X$ be a K3 surface or an Enriques surface over an algebraically closed field of characteristic different from $2$, or a projective irreducible holomorphic symplectic manifold. 
\begin{enumerate}
\item
An automorphism $g\in \Aut(X)$ of positive entropy is achiral if and only if $\overline{e}(g)=0$. 
\item 
There exists a positive constant $\Theta$ 
such that $\overline{e}(g)\ge \Theta$ for any chiral automorphism $g\in\Aut(X)$ of positive entropy. 
\end{enumerate}
\end{athm}


The key ingredients in the proof of \cref{Gap_thm: ent_Aut-intro} are as follows:  
\begin{itemize}
\item The geometrical finiteness of $\Aut(X)$, established in \cite{Kikuta,Takatsu}.
\item The acylindrical action of a geometrically finite Kleinian group on a certain hyperbolic graph, 
constructed in \cite{Osin}. 
\item A generalization of the arguments in \cite{Cal-Fuj10} to groups admitting acylindrical actions 
on a hyperbolic space, developed in \S\ref{subection: main-estimate}. 
\end{itemize}

By a similar argument, we also prove a gap theorem for the entropy norm on mapping class groups by using the action on the curve graph; see \cref{Gap_thm: ent_MCG}.

\subsubsection{Achirality of parabolic automorphisms}
In this subsection, we consider the case where $X$ is a K3 surface or an Enriques surface.
Automorphisms of finite order are clearly achiral, 
and the achirality of positive-entropy automorphisms can be described in terms of the entropy norm by \cref{Gap_thm: ent_Aut-intro}. 
Accordingly,  we focus on 
automorphisms of infinite order and zero entropy (equivalently, parabolic automorphisms).
It is known that such an automorphism preserves a genus-one fibration. 
To state our theorems, we recall some terminology; see \S\ref{sec:para} for details. 
Let $X$ be a K3 surface or an Enriques surface
over an algebraically closed field, and 
let $p:X\to B$ be a genus-one fibration. 
We say that an automorphism $g$ of $X$ \emph{preserves} $p$ if  
there exists $b(g):B\to B$ such that $b(g)\circ p = p \circ g$. 
The group of automorphisms of $X$ that preserve $p$ is denoted by $\Aut_p(X)$. 
One associates to $p$ another genus-one fibration $J_p: J_X \to B$ called 
the \emph{Jacobian fibration}. The Jacobian fibration admits a canonical section. 
We write $[-1]\in \Aut_{J_p}(J_X)$ for the inversion involution with respect to 
the canonical section. There exists a natural homomorphism 
\[ \varphi:\Aut_p(X) \to \Aut_{J_p}(J_X) \]
defined by the push-forward of line bundles. 

We say that a subset $A$ of a group $\Gamma$ is \emph{uniformly achiral} if 
there exists $b\in \Gamma$ such that for every $a\in A$, there exists 
$m\in \ZZ_{>0}$ satisfying $ba^mb^{-1} = a^{-m}$. 


\begin{athm}[\cref{prop:exsistenceof-1,th:para_K3}]
Let $X$ be a K3 surface over an algebraically closed field,
and let $p: X \rightarrow B$ be an elliptic fibration of positive Mordell--Weil rank.
Then the following conditions are equivalent:
\begin{enumerate}
\item The image of $\varphi:\Aut_p(X)\to \Aut_{J_p}(J_X)$ contains $[-1]$. 
\item The multisection index of $p$ is either $1$ or $2$. 
\item There exist $g, h \in \Aut_p(X)$ such that $g$ is parabolic, 
$\ord b(h) \neq \ord \varphi(h)$, and $h \circ g \circ h^{-1} = g^{-1}$. 
\end{enumerate}
Moreover, if one of these conditions holds, then $\Aut_p(X)$ is uniformly achiral. 
\end{athm}

In the case of K3 surfaces, there are examples in which $\Aut_p(X)$ contains both achiral parabolic automorphisms and chiral parabolic automorphisms (see \S\ref{ss:Examples}).
We obtain the following theorem for Enriques surfaces.

\begin{athm}[\cref{th:para-auto-of-Enriques}]
Let $X$ be an Enriques surface over an algebraically closed field. 
Then $\Aut_p(X)$ is uniformly achiral for each genus-one fibration $p:X\to B$. 
In particular, every parabolic automorphism of $X$ is achiral.  
\end{athm}

\subsection{Organization of the paper}

In Section 2, we collect basic material from geometric group theory and algebraic geometry.
In Section 3, we prove the gap theorem for the entropy norm on automorphism groups of algebraic varieties (\cref{Gap_thm: ent_Aut-intro}). 
In the course of the proof, we also establish gap theorems for general stable norms and for $\scl$ on groups admitting acylindrical actions, 
as well as a gap theorem for the stable entropy norm on mapping class groups.
In Section 4, after reviewing some algebraic geometry of genus-one curves, 
we study the achirality of parabolic automorphisms of K3 surfaces and Enriques surfaces using genus-one fibrations and their Jacobian fibrations. 
In addition, in the case of K3 surfaces, we construct several examples concerning achirality by means of lattice theory.

\begin{ack}
The authors would like to thank Simon Brandhorst, Morimichi Kawasaki, Mitsuaki Kimura, Shuhei Maruyama, Shin-ichi Oguni and Hisanori Ohashi for valuable discussions and helpful comments. 

K. K. is indebted to Arend Bayer and the University of Edinburgh for their kind hospitality. 
K. K. is supported by JSPS Overseas Research Fellow and JSPS KAKENHI Grant Number 21K13780. 

A major part of this work was carried out while Y. T. was a visiting researcher at Saarland University. 
He would like to express sincere gratitude to the university and his host, Professor Simon Brandhorst, 
for their warm hospitality and for providing an excellent research environment. 
He is also grateful to Reinder Meinsma for fruitful discussions about elliptic fibrations. 
Y. T. is supported by JSPS KAKENHI Grant Number JP24KJ0044.

T. T. would like to express sincere gratitude to Professor Hisanori Ohashi for his kind hospitality and for providing a stimulating research environment.
T. T. is supported by JSPS KAKENHI Grant Number
JP25KJ0328.

\end{ack}

\section{Preliminaries}

\subsection{Hyperbolic spaces}
We recall hyperbolic spaces here. We refer to 
\cite{Gromov_HyperbolicGroups, Bridson-Haefliger1999} for details. 

Let $\delta>0$. A geodesic triangle in a metric space is \emph{$\delta$-slim}  
if each of its sides is contained in the $\delta$-neighborhood of the union 
of the other two sides. 
A geodesic space $S$ is \emph{$\delta$-hyperbolic} if every geodesic triangle 
in $S$ is $\delta$-slim. A \emph{hyperbolic space} is a $\delta$-hyperbolic space for some $\delta$. 

Let $(S,d)$ be a metric space. 
For a real number $r \geq 0$ and a subset $A\subset S$, 
we write $V(A,r)$ for the $r$-neighborhood $\{x\in S \mid  d(x, A) \leq r \}$ of $A$.  
The \emph{Hausdorff distance} between subsets $A$ and $B$ of $S$ is defined by 
\[ 
d_H(A,B) \coloneq \inf\{ r>0 \mid A \subset V(B,r) \text{ and } B \subset V(A,r) \}. 
\]
For $\lambda\geq 1$ and $\epsilon\geq 0$, a \emph{$(\lambda, \epsilon)$-quasi-geodesic} 
in $S$ is a path $c:I\to S$, where $I$ is an interval of the real line, 
such that 
\[ \frac{1}{\lambda}|t-t'| - \epsilon
\leq d(c(t), c(t'))
\leq \lambda|t-t'| + \epsilon \quad(t,t'\in I). 
\]
A $(\lambda, \epsilon)$-quasi-geodesic is referred to as a quasi-geodesic 
when the specific values of the constants are unnecessary. 
The following theorem is called the Morse lemma and is fundamental.  

\begin{thm}\label{MorseLemma}
Let $\delta >0$, $\lambda\geq 1$, and $\epsilon\geq0$. There exists a constant
$L = L(\delta, \lambda, \epsilon)$ with the following property: 
for any $\delta$-hyperbolic space $S$ and any $(\lambda, \epsilon)$-quasi-geodesics $c$ and $c'$ 
with the same endpoints in $S$, the Hausdorff distance between $\im c$ and $\im c'$ is at most $L$.  
\end{thm}
\begin{proof}
See \cite[Theorem III.H-1.7]{Bridson-Haefliger1999}. 
\end{proof}

We will refer to the constant $L = L(\delta, \lambda, \epsilon)$ as a \emph{Morse constant}. 

Recall that the \emph{translation length} of an isometry $g$ of $S$ is defined by 
$\tau(g) \coloneq \lim_{n\to \infty}\frac{1}{n}d(x, g^nx)$, where $x\in S$. 
This limit exists by Fekete's lemma and equals $\inf_{n> 0} \frac{1}{n}d(x, g^n x)$. 
Furthermore, it does not depend on the choice of $x\in S$. 
We say that $g$ is 
\begin{itemize}
\item \emph{elliptic} if an orbit $\{ g^nx \mid n\in \ZZ \}$ is bounded. 
\item \emph{parabolic} if an orbit $\{ g^nx \mid n\in \ZZ \}$ is unbounded, and 
the translation length $\tau(g)$ is zero. 
\item \emph{loxodromic} if the translation length $\tau(g)$ is positive. 
\end{itemize}

\subsection{Norms and quasimorphisms}
Let $\Gamma$ be a group. The identity element of $\Gamma$ is denoted by $1$. 
A \emph{norm} of $\Gamma$ is a function $\nu:\Gamma\to \RR_{\geq0}$ with the 
following properties: 
\begin{enumerate}
\item $\nu(g) = 0$ if and only if $g = 1$,
\item $\nu(gh) \leq \nu(g) + \nu(h)$ for any $g,h\in \Gamma$, 
\item $\nu(g^{-1}) = \nu(g)$ for any $g\in \Gamma$. 
\end{enumerate}
For a norm $\nu$, the \emph{stabilization} $\overline{\nu}:\Gamma\to \RR_{\ge0}$ of $\nu$ is 
the function defined by 
\[ \overline{\nu}(g) \coloneqq \lim_{n\to\infty} \frac{\nu(g^n)}{n} \quad(g\in \Gamma).  \] 
Note that $\overline{\nu}$ is not a norm in general, 
but we call $\overline{\nu}$ a \emph{stable norm}. 
One sees that a stable norm is invariant under conjugation. 

A subset $A$ of $\Gamma$ is said to be \emph{symmetric} if $a^{-1} \in A$ for any $a\in A$. 
Let $A$ be a symmetric subset of $\Gamma$. 
Then, we can define a norm $\nu_A:\langle A\rangle\to \ZZ_{\geq0}$ by 
\[ \nu_A(g) \coloneqq (\text{the least number of elements of $A$ whose product is equal to $g$}). \]
The norm $\nu_A$ is called the \emph{word norm associated with $A$}.  
For later use, we naturally extend the stable norm $\overline{\nu_A}$ to 
\[
\overline{\nu_A}:\Gamma\to \RR_{\geq0}\cup\{\infty\}
\]
as follows: 
if $g^m\in\langle A\rangle$ for some positive integer $m$, define
$\overline{\nu_A}(g) \coloneq \overline{\nu_A}(g^m)/m$, 
and if no power of $g$ is contained in $\langle A \rangle$, 
define $\overline{\nu_A}(g) \coloneq \infty$. 

We see that a word norm can be estimated from below by using a quasimorphism, 
which is defined below. 

\begin{dfn}
A function $\psi: \Gamma\to \RR$ is called a \emph{quasimorphism} if 
there exists a real number $D\geq 0$ such that 
$|\psi(gh) - \psi(g) - \psi(h)| \leq D$ for any $g,h\in \Gamma$. 
Let $\psi$ be a quasimorphism. We call the number 
\[ D_\psi\coloneqq\inf \{D\geq 0 \mid |\psi(gh) - \psi(g) - \psi(h)| \leq D 
\text{ for all $g,h\in \Gamma$} \} 
\]
the \emph{defect} of $\psi$. 
We say that $\psi$ is \emph{homogeneous} if 
$\psi(g^n) = n\psi(g)$ for any $g$ and $n\in\ZZ$. 
The \emph{homogenization} of $\psi$ is the function $\overline{\psi}:\Gamma\to \RR$ 
defined by 
\[ \overline{\psi}(g) \coloneqq \lim_{n \to \infty} \frac{\psi(g^n)}{n} \quad (g\in \Gamma).  \]
The limit exists, and $\overline{\psi}$ is a homogeneous quasimorphism.
\end{dfn}

\begin{prop}\label{basic-ineq}
Let $\Gamma$ be a group and $A$ a symmetric subset of $\Gamma$. 
Let $\psi:\Gamma\to \RR$ be a homogeneous quasimorphism with defect $D_\psi$. 
Suppose that there exists a constant $C > 0$ such that 
$\sup\{ |\psi(a)| \mid a\in A \} \leq C$. 
Then 
\[ \overline{\nu_A}(g) \geq \frac{|\psi(g)|}{D_\psi + C} \]
for any $g\in \Gamma$. 
\end{prop}
\begin{proof}
This is a consequence of easy computations; see \cite[Section 1]{Kot04}. 
\end{proof}

\subsection{Acylindrical actions and WPD elements}
\begin{dfn}[\cite{Bowditch2008, Osin}]
The action of a group $\Gamma$ on a metric space $(S,d)$ is said to be
\emph{acylindrical} if for every $r > 0$ there exist $R, N > 0$ such that 
for every two points $x, x' \in S$ with $d(x,x')\geq R$, 
\[ \# \{ g\in \Gamma \mid \text{$d(x,gx) \leq r$ and $d(x,gx') \leq r$} \}\leq N. \]
\end{dfn}
The following theorem is known. 

\begin{thm}[{\cite[Lemma 2.2]{Bowditch2008}, \cite[Lemma 6.28]{DGO17}}]\label{th:t_0-is-positive}
Let $\Gamma$ be a group acting acylindrically on a $\delta$-hyperbolic space. 
Then each element of $\Gamma$ is either elliptic or loxodromic. 
Moreover, there exists $\eta>0$ depending only on $\delta$ and the parameters of acylindricity 
such that $\tau(g) > \eta$ for any loxodromic element $g\in \Gamma$. 
\end{thm}

\begin{dfn}[\cite{BF02}]\label{def: WPD}
Let $\Gamma$ be a group acting on a hyperbolic space $(S,d)$. 
An element $g\in\Gamma$ satisfies the \emph{weak proper discontinuity} condition,
or $g$ is a \emph{WPD element}, 
if for every $x\in S$ and every $r>0$, there exists $n\in\ZZ_{>0}$ such that
the set
\[ \{a\in\Gamma \mid \text{$d(x,ax)\leq r$ and $d(g^nx,ag^nx) \leq r$}\} \]
is finite. 
\end{dfn}

In the setting of \cref{def: WPD}, every element of $\Gamma$ is WPD 
if the action is proper. If the action is acylindrical, then 
every loxodromic element of $\Gamma$ is WPD. 


\subsection{Hyperbolic spaces associated with surfaces}\label{ss:surfaces}
We recall fundamental facts on algebraic surfaces, 
especially K3 surfaces and Enriques surfaces. 
We refer to \cite{Cossec-Dolgachev-Liedtke_EnriquesI} for details. 

Let $k$ be an algebraically closed field, 
and let $X$ be a smooth projective surface over $k$. 
The group $\Num(X)$ of divisor classes modulo numerically equivalence 
is a finitely generated free abelian group. 
The rank of $\Num(X)$ is called the \emph{Picard number} of $X$ and is denoted 
by $\rho(X)$. The intersection form $(\cdot, \cdot):\Num(X)\times \Num(X)\to \ZZ$
endows $\Num(X)$ with the structure of a lattice of signature $(1, \rho(X)-1)$. 

Every automorphism $g$ of $X$ acts isometrically on $\Num(X)$ via the push-forward $g_*$. 
We write $\varrho:\Aut(X)\to \mathrm{O}(\Num(X))$ for this representation. 

We say that $X$ is a \emph{K3 surface} if $\Omega_{X/k}^2 \cong \O_X$ and $H^1(X, \O_X) = 0$, 
and that $X$ is an \emph{Enriques surface} if 
its Kodaira dimension is zero and its second Betti number is $10$. 

\begin{prop}\label{prop:finite_kernel}
Suppose that $X$ is a K3 surface or an Enriques surface. 
Then, the kernel of $\varrho:\Aut(X)\to \mathrm{O}(\Num(X))$ is finite. 
\end{prop}
\begin{proof}
See \cite[Proposition 8.2.1]{Dolgachev-Kondo_EnriquesII}. 
\end{proof}

In the following, we write $L = \Num(X)$ and $\rho = \rho(X)$. 
Let $V \coloneq L\otimes \RR$. The intersection form extends linearly to $V$. 
The set $\{ v\in V \mid (v,v)>0 \}$ has two connected components. 
The one that contains an ample class is called the \emph{positive cone} and denoted by $\C$.
Then, the hypersurface
\[ \HH_X \coloneq \{ v\in \C \mid (v,v) = 1 \} \]
endowed with the metric $d$ determined by $\cosh d(u,v) = (u,v)$ is a hyperbolic space. 
The automorphism group $\Aut(X)$ of $X$ acts on $\HH_X$ by isometries via push-forward. 

\begin{dfn}
Let $g$ be an automorphism of $X$. We say that 
$g$ is \emph{elliptic}, \emph{parabolic}, or \emph{loxodromic} if the induced isometry
$g_*:\HH_X \to \HH_X$ is elliptic, parabolic, or loxodromic, respectively. 
\end{dfn}

Let $g$ be an automorphism of $X$. The \emph{homological entropy}, 
or simply the entropy, of $g$ is defined to be 
$\log (\lambda(g_*))$, where $\lambda(g_*)$ denotes 
the spectral radius of the linear transformation $g_*:V\to V$. 
By the results of Gromov \cite{Gromov_Entropy} and Yomdin \cite{Yomdin1987}, 
if $k = \CC$, then the homological entropy of $g$ 
coincides with the topological entropy of $g$ with respect to the Euclidean topology on $X$. 

\begin{lem}\label{lem:tl_equals_ent}
Let $g$ be an automorphism of $X$. Then, the entropy of $g$  
is equal to the translation length $\tau(g_*:\HH_X\to \HH_X)$.  
\end{lem}
\begin{proof}
This follows from \cite[Exercise 4.7-26]{Rat}. 
\end{proof}

A proper and surjective morphism $p: X\to B$ of varieties over $k$ with $p_*\O_X = \O_B$
is called a \emph{fibration}. 
We say that an automorphism $g$ of $X$ \emph{preserves} a fibration $p:X\to B$ if  
there exists $b(g):B\to B$ such that $b(g)\circ p = p \circ g$. 
Let $p:X\to B$ be a fibration from a smooth surface $X$ to a smooth curve $B$ over $k$. 
Let $K$ be the function field of $B$, and let $X_K$ denote the generic fiber of $p$. 
Then, $X_K$ is a geometrically integral and regular curve over $K$ 
(\cite[Theorem 4.1.1]{Cossec-Dolgachev-Liedtke_EnriquesI}). 
We say that $p$ is a \emph{genus-one fibration} if 
the \emph{arithmetic genus} $p_a(X_K) = \dim_K H^1(X_K, \O_{X_K})$ is equal to $1$. 
A genus-one fibration $p:X\to B$ is called an \emph{elliptic fibration} if 
$X_K$ is smooth, and a \emph{quasi-elliptic fibration} otherwise. 
The following proposition is known. 

\begin{prop}\label{prop:trichotomy}
Let $X$ be a K3 surface or an Enriques surface over $k$, and 
let $g$  be an automorphism of $X$. 
\begin{enumerate}
\item $g$ is elliptic if and only if $g$ is of finite order. 
\item $g$ is parabolic if and only if $g$ is of infinite order and preserves a genus-one fibration. 
\item $g$ is loxodromic if and only if the entropy of $g$ is positive. 
\end{enumerate}
\end{prop}
\begin{proof}
Claims (i) and (iii) follow from \cref{prop:finite_kernel} and 
\cref{lem:tl_equals_ent}, respectively. We prove (ii). 
Suppose that $g$ is of infinite order. We remark 
that $g$ is parabolic if and only if $g_*:V\to V$ fixes an isotropic vector in $L$
(see e.g. \cite[Exercise 4.7-18]{Rat}). 
If $g$ preserves a genus-one fibration, then its fiber class in $L$ is an  
isotropic vector fixed by $g_*$, and hence $g$ is parabolic. 
Conversely, suppose that $g$ is parabolic. Then, it follows from 
\cite[Proposition 2.2.1]{Cossec-Dolgachev-Liedtke_EnriquesI} that 
$g$ fixes an isotropic vector $f\in L$ that is nef. 
We may assume that $f$ is primitive in $L$. Then 
it follows from \cite[Proposition 2.3.10]{Huybrechts_K3book} 
(resp. \cite[Propositions 2.3.3 and 2.2.8]{Cossec-Dolgachev-Liedtke_EnriquesI}) 
that the linear system $|f|$ (resp. $|2f|$) defines a genus-one fibration, 
and it is preserved by $g$ in the case of K3 surfaces (resp. Enriques surfaces).  
The proof is complete. 
\end{proof}

\subsection{Irreducible holomorphic symplectic manifolds}\label{ss:IHS}
An \emph{irreducible holomorphic symplectic manifold} 
(abbreviated as an \emph{IHS manifold}) 
is a simply connected compact K\"ahler manifold $X$ such that 
$H^0(X, \Omega_X^2)$ is generated by an everywhere nondegenerate holomorphic $2$-form. 
K3 surfaces (over $\CC$) are examples of IHS manifolds, and IHS manifolds are regarded as 
higher-dimensional analogues of K3 surfaces. 
We observe that the automorphism group of an IHS manifold acts on a hyperbolic space by isometries
in a natural way, as in the case of surfaces discussed in \S \ref{ss:surfaces}. 
We refer to \cite{Huybrechts_CptHKMfds} for details. 

Let $X$ be an IHS manifold. 
Then the second cohomology group $H^2(X, \ZZ)$ is torsion-free and admits an inner product 
$(\cdot, \cdot):H^2(X, \ZZ)\times H^2(X, \ZZ)\to \ZZ$
called the \emph{Beauville--Bogomolov--Fujiki form}. 
The signature of the form is $(3, b_2(X) - 3)$, 
where $b_2(X)$ is the second Betti number. Moreover, 
the real part of $H^{2,0}(X) \oplus H^{0,2}(X)$ is a positive-definite
$2$-dimensional subspace of $H^2(X,\RR) = H^2(X, \ZZ)\otimes \RR$. 
The subgroup $\NS(X) \coloneq H^2(X, \ZZ) \cap H^{1,1}(X)$ of $H^2(X, \ZZ)$ is called 
the \emph{N\'eron--Severi lattice}. Let $\rho(X) = \rk \NS(X)$. 
If $X$ is projective, then the signature of $\NS(X)$ is given by $(1, \rho(X)- 1)$. 
In fact, the converse also holds (see \cite[\S26]{Huybrechts_CptHKMfds}). 

Let $g$ be an automorphism of $X$. Then
$g_* \coloneq (g^*)^{-1}: H^2(X,\ZZ)\to H^2(X,\ZZ)$
preserves the Beauville--Bogomolov--Fujiki form. 
We define the \emph{(homological) entropy} of $g$ to be 
\[ \frac{\dim_{\CC} X}{2} \log(\lambda(g_*)),\] 
where $\lambda(g_*)$ denotes the spectral radius of $g_*:H^2(X, \ZZ)\to H^2(X, \ZZ)$. 
By the Gromov--Yomdin theorem and \cite[Theorem 1.1]{Ogu09}, 
the homological entropy of $g$ coincides with the topological entropy of $g$. 

Assume that $X$ is projective. Let $L \coloneq \NS(X)$ and $V \coloneq L\otimes \RR$. 
Then, as in \S \ref{ss:surfaces}, we define the positive cone $\C$ to be  
the connected component of $\{ v\in V \mid (v,v)>0 \}$ that contains an ample class, 
and we associate to $X$ the hyperbolic space 
\[ \HH_X \coloneq \{ v\in \C \mid (v,v) = 1 \}.  \]
The automorphism group $\Aut(X)$ acts on $\HH_X$ by isometries via $g\mapsto g_*$.  
As in the case of surfaces, we say that an automorphism
$g$ of $X$ is \emph{elliptic}, \emph{parabolic}, or \emph{loxodromic} if the induced isometry
$g_*:\HH_X \to \HH_X$ is elliptic, parabolic, or loxodromic, respectively.

\begin{lem}\label{lem:tl_and_ent:HK}
Let $g$ be an automorphism of a projective IHS manifold $X$. 
Then the entropy of $g$ is equal to $\frac{\dim_\CC X}{2} \tau(g_*)$, where 
$\tau(g_*)$ denotes the translation length of $g_*:\HH_X\to \HH_X$.  
In particular, $g$ is loxodromic if and only if its entropy is positive. 
\end{lem}
\begin{proof}
In this case, the spectral radius of $g_*:\NS(X)\to\NS(X)$ 
coincides with that of $g_*:H^2(X,\ZZ)\to H^2(X,\ZZ)$; 
see e.g. \cite[Theorem 4.1]{Keum-Oguiso-Zhang2009}.  
Then, the statement follows from \cite[Exercise 4.7-26]{Rat}. 
\end{proof}

\section{Gap theorems}\label{Section: gap-thm}

In this section, we prove the gap theorems for several stable norms. 

\subsection{Achiral elements}
Let $\Gamma$ be a group. 
We begin by introducing the following notion: 
\begin{dfn}\label{def-mirror-condition}
An element $a\in\Gamma$ is said to be \emph{achiral}
if $a^m$ is conjugate to $a^{-m}$ for some $m\in\ZZ_{>0}$. 
Otherwise, $a$ is \emph{chiral}.
\end{dfn}

Let $A$ be a symmetric subset of $\Gamma$ containing every element of finite order. 

\begin{lem}\label{achiral-implies}
For an achiral element $a\in\Gamma$, 
either 
some power of $a$ lies in a subgroup of $\Gamma$ isomorphic to $\ZZ^2$, or 
$\overline{\nu_A}(a) = 0$.  
\end{lem}
\begin{proof}
Let $a\in \Gamma$ be an achiral element: 
$ba^mb^{-1}=a^{-m}$ for some $b\in\Gamma$ and $m\in\ZZ_{>0}$.  
If the order of $a$ is finite then $\nu_A(a^k)\leq 1$ for any $k\in\ZZ$, hence 
$\overline{\nu_A}(a) = 0$.  
Suppose that $a$ has infinite order. Note that $a^m$ and $b^2$ are commutative and that 
$b^2$ is not a power of $a^m$.  
If $b^2$ has infinite order then the subgroup $\langle a^m, b^2\rangle$ is isomorphic to 
$\ZZ^2$, and we are done. 
Assume that the order of $b^2$ is finite, say $d$. For any $k\in \ZZ$, we have 
\[ (a^{mk}ba^{-mk}b)^d 
= (a^{2mk}b^2)^d
= a^{2mkd}b^{2d}
= a^{2mkd}. 
\]
Since $a^{mk}ba^{-mk}$ and $b$ are of finite order,
we note that, at this point, the case of $\overline{\nu_A}(a)=\infty$ has been ruled out.
We then obtain
\[ \frac{1}{k}\nu_A((a^{2md})^k)
= \frac{1}{k}\nu_A((a^{mk}ba^{-mk}b)^d)
\le \frac{2d}{k}
\xrightarrow{k\to \infty} 0, 
\]
hence $\overline{\nu_A}(a) = 0$. 
\end{proof}

The following is a part of our gap theorems. 

\begin{prop}\label{achiral-WPD-normzero}
Suppose that $\Gamma$ acts by isometries on a hyperbolic space. 
For any achiral loxodromic WPD element $a\in \Gamma$, 
we have $\overline{\nu_A}(a)=0$.  
\end{prop}
\begin{proof}
Since the centralizer of a loxodromic WPD element is virtually cyclic \cite[Proposition 6]{BF02}, 
we have $\overline{\nu_A}(a) = 0$ by \cref{achiral-implies}.
\end{proof}

\subsection{Uniform estimate for quasimorphisms}
\label{subection: main-estimate}
Here we prove a uniform estimate (\cref{th:uniform-lower-bd}) for quasimorphisms, 
which is a key ingredient in the proofs of the gap theorems. 

Let $(S,d)$ be a $\delta$-hyperbolic space, and let $\Gamma$ be a group acting on $S$. 
We begin by recalling the construction of Epstein--Fujiwara quasimorphisms from \cite{Fuj98} 
(see also \cite[\S 4.1]{Huangfu-Wan}).  
A \emph{copy} of a path $\alpha$ in $S$ is a path of the form $g\alpha$ for some $g\in\Gamma$. 
For a path $w$ of finite length $|w|$, a real number $W$ with $0<W<|w|$,  
and a geodesic $[x,y]$ in $S$ from $x$ to $y$, we define
\begin{equation}\label{def_count_func}
c_{w,W}([x,y]) \coloneqq d(x,y)-\inf_\alpha\{|\alpha|-W|\alpha|_w\}\in\RR_{\ge0}, 
\end{equation}
where $\alpha$ ranges over all the paths from $x$ to $y$, and $|\alpha|_w$ is 
the maximal number of disjoint oriented copies of $w$ which can be obtained as nontrivial subpaths of $\alpha$. 
Note that (\ref{def_count_func}) does not depend on the choice of the geodesic from $x$ to $y$. 

Since a path realizing the infimum in (\ref{def_count_func}) does not exist in general (\cite[Remark 4.3]{Huangfu-Wan}), 
Huangfu--Wan \cite{Huangfu-Wan} introduced the following notion: 
a path $\beta$ between $x$ and $y$ is \emph{almost realizing path} of $c_{w, W}([x,y])$ if it satisfies
\[
d(x,y)-(|\beta|-W|\beta|_w) 
\ge \max\left\{ 
c_{w,W}([x,y])-W, ~ 
\frac{c_{w, W}([x,y])}{2}
\right\}. 
\]
Then, an almost realizing path always exists, 
and if $|\beta|_w=0$ for an almost realizing path $\beta$ then $c_{w,W}([x,y])=0$
 (\cite[\S 4.1]{Huangfu-Wan}). 

\begin{lem}[{\cite[Lemma 4.5]{Huangfu-Wan}}]\label{Fuji98_Lemma3.3}
Let $\beta$ be an almost realizing path of $c_{w, W}([x,y])$. 
Then $\beta$ is $(\frac{|w|}{|w|-W}, 3W)$-quasi-geodesic.
\end{lem}

Let $x_0\in S$ be a point. For $a\in \Gamma$, we define
\[  
h_{w,W, x_0}(a) \coloneq c_{w,W}([x_0, a x_0]) - c_{w^{-1}, W}([x_0, a x_0]). 
\]

\begin{prop}[{\cite[Proposition 3.10]{Fuj98} and \cite[Lemma 4.6]{Huangfu-Wan}}]\label{Fuj98_Prop3.10}
The map $h_{w, W, x_0}: \Gamma\to \RR$ is a quasimorphism with defect at most 
$12L + 6W + 48\delta$, where $L = L(\delta, \frac{|w|}{|w|-W}, 3W)$ is a Morse constant. 
\end{prop}

Using the Epstein--Fujiwara quasimorphism defined above,  
we will associate a quasimorphism $h_g$ to each loxodromic element $g\in \Gamma$. 
We need the following lemma. 

\begin{lem}[{\cite[Lemma 6.27]{DGO17}}]\label{lem:TandK}
There exist constants $T, K>0$ depending only on $\delta$ such that 
the following holds: for any $g\in \Gamma$ with $\tau(g)\geq T$, any $x\in S$ with 
$d(x, gx) - \inf_{z\in S}{d(z, gz)} \leq \delta$, and any $i\in \ZZ_{\geq 0}$, we have 
\[ i\tau(g) \leq d(x,g^ix) \leq i \tau(g) + K.  \]
\end{lem}

In the following, we fix the constants $T$ and $K$ in \cref{lem:TandK}. 
Let $g\in \Gamma$ with $\tau(g)\geq 2K+2$, and take a point $x_0\in S$. 
Then $d(x_0, gx_0) \geq \tau(g) \geq 2K+2$. 
Put $W \coloneq K + 1$. Then 
$\frac{d(x_0, g x_0)}{d(x_0, g x_0)-W} \leq 2$, and
any almost realizing path of $c_{[x_0, g x_0], W}$ 
is a $(2, 3K + 3)$-quasi-geodesic by \cref{Fuji98_Lemma3.3}. 
We fix a Morse constant $L = L(\delta, 2, 3K+3)$. 
We write 
\[ h_{g} = h_{[x_0, g x_0], K+1, x_0} \] 
for simplicity.

\begin{prop}\label{prop:upper-bdd-of-defect}
The defect of the homogenization $\overline{h_g}$ of $h_g$ is at most 
$D \coloneq 24 L + 12(K+1) + 96\delta$, which depends only on $\delta$. 
\end{prop}
\begin{proof}
By \cref{Fuj98_Prop3.10}, it follows  that $D_{h_g} \leq 12 L + 6(K+1) + 48\delta$. 
On the other hand, it is known that $D_{\overline{\psi}} \leq 2D_\psi$ 
for any quasimorphism $\psi$; see \cite[Corollary 2.59]{scl}. 
Therefore, $D_{\overline{h_g}} \leq 2D_{h_g} \leq D $ as desired. 
\end{proof}

The goal of this subsection is to prove the following theorem. 

\begin{thm}\label{th:uniform-lower-bd}
Let $\Gamma$ be a group acting acylindrically on $S$. 
Then there exists $M \in \ZZ_{>0}$, depending only on $\delta$ and  
the parameters of acylindricity, such that 
$\overline{h_{g^M}}(g) \geq \frac{1}{M}$ for any chiral loxodromic element $g\in \Gamma$. 
\end{thm}

A theorem similar to \cref{th:uniform-lower-bd} was proved in \cite{Cal-Fuj10} in the cases 
of hyperbolic groups and mapping class groups. 
In our setting, a loxodromic element need not admit a geodesic axis, whereas 
such an axis always exists in the situations considered in \cite{Cal-Fuj10} 
and is used in their proof. 
To overcome the lack of axes, we instead use a quasi-geodesic invariant under 
a given loxodromic element, with constants independent of the element.  

For $g\in \Gamma$, we define $l_g \coloneq \bigcup_{i\in \ZZ}g^i [x_0,gx_0]$ 
and parametrize $l_g$ by arc length so that $l_g(0)= x_0$, where 
$x_0\in S$ is a point with $d(x_0, g x_0) - \inf_{z\in S}d(z, gz) \leq \delta$. 

\begin{lem}\label{lem:l_g}
For $g\in \Gamma$ with $\tau(g)\geq \max\{T, 2\}$, the path $l_g$ is a $(2, 2K)$-quasi-geodesic.  
\end{lem}
\begin{proof}
Let $d_0 = d(x_0, gx_0)$. Let $t, t'\in \RR$ with $t'>t$, and 
define integers $i$ and $n$ so that 
$d_0 i \leq t < d_0(i + 1)$ and $d_0(i + n - 1) \leq t' < d_0(i + n)$. 
Then, by \cref{lem:TandK}, we have 
\[\begin{split}
d(l_g(t), l_g(t'))
&\geq d(g^i x_0, g^{i+n} x_0) - d(g^ix_0, l_g(t)) - d(g^{i+n}x_0, l_g(t')) \\
&= d(x_0, g^n x_0) - (t - d_0 i) - (d_0(i+n) - t') \\
&\geq n\tau(g) + (t' - t)  - n(\tau(g) + K) \\
&= (t' - t)  - nK.
\end{split}
\]
On the other hand, we have 
\[ t' - t \geq (n-2)d(x_0,gx_0) \geq (n-2)\tau(g).   \]
Hence  
\[\begin{split}
d(l_g(t), l_g(t'))
&\geq (t' - t)  - nK \\
&= (t' - t)  - (n-2)K - 2K \\
&\geq (t' - t)  - \frac{t' - t}{\tau(g)}K - 2K \\
&= \frac{\tau(g) - 1}{\tau(g)}(t' - t) - 2K,  
\end{split}
\]
which shows that $\alpha$ is a $(\frac{\tau(g)}{\tau(g) - 1}, 2K)$-quasi-geodesic. 
Since $\tau(g) \geq 2$, we have $\frac{\tau(g)}{\tau(g) - 1} \leq 2$, and the proof is complete.  
\end{proof}

The following proposition is a generalization of the claim in \cite{Cal-Fuj10}
that is referred to as ``no reverse counting"  
and play an important role in the proof of \cref{th:uniform-lower-bd}. 

\begin{prop}\label{prop:key-for-uniform-lower-bd}
Suppose that the action of $\Gamma$ on $S$ is acylindrical.
Let $R>0$ and $N\in \ZZ_{>0}$ be constants such that  
\[ \#\{a\in G\mid d(x,ax) \leq 10L+K, d(x',ax') \leq 10L+K \} \leq N  \]
for any $x,x'\in S$ with $d(x,x')\geq R$. 
Let $g\in \Gamma$ with $\tau(g) \geq \max\{ T, 2 , R + 3L\}$, and let
$x_0\in S$ with $d(x_0, g x_0) - \inf_{z\in S}d(z, gz) \leq \delta$.
If there exists a copy of $[x_0, g^{-(N+1)} x_0]$ 
contained in the $L$-neighborhood $V(l_g, L)$ of $l_g$ 
whose orientation agrees with that of $l_g$, then $g$ is achiral. 
\end{prop}
\begin{proof}
Let $\beta$ be a copy of $[x_0, g^{-(N+1)} x_0]$ in $V(l_g, L)$ whose orientation agrees with 
that of $l_g$, and write $\beta = h[x_0, g^{-(N+1)}x_0]$ for some $h \in \Gamma$. 
Let $e = hg^{-1}h^{-1}\in \Gamma$, and set $y_0 = hx_0$ and $y_0'  = ey_0\in S$. Note that 
$e^i y_0 = h g^{-i}x_0$ lies on the path $hl_g$ for any $i\in \ZZ$. 
Let $\gamma$ be the subpath of $hl_g$ from $y_0$ to $e^{N+1} y_0$. 
Then $\gamma$ is a $(2, 2K)$-quasi-geodesic 
(and hence a $(2, 3K+3)$-quasi-geodesic) by \cref{lem:l_g}. 
Since $\beta = [y_0, e^{N+1} y_0]$, we have $d_H(\beta, \gamma) \leq L$ by \cref{MorseLemma}.  

Let $z_0$ and $z_0'$ be nearest-point projections onto $l_g$ of $y_0$ and $y_0'$, respectively. Then 
\[ \begin{split}
d(z_0, z_0') 
&\geq d(y_0, y_0') - d(y_0, z_0) - d(y_0', z_0') \\
&\geq d(y_0, e y_0) - L - 2L \\
&\geq \tau(g) - 3L \\
&\geq R.
\end{split}
\]
Now, we prove that $d(z_0, e^{-i}g^i z_0) \leq 10L + K$ for any $i \in \{0,1,\ldots, N\}$. 
Let $i \in \{0,1,\ldots, N\}$. 
Let $z_i$ be a nearest-point projection of $e^iy_0$ onto $[z_0, g^i z_0]$. 
Note that the Hausdorff distance between $[z_0, g^i z_0]$ and 
the subpath of $l_g$ from $z_0$ to $g^i z_0$ is at most $L$ by \cref{MorseLemma}. 
Hence,  
\[ \begin{split}
d(g^iz_0, e^iz_0) 
&\leq d(g^iz_0, z_i) + d(z_i, e^iy_0) + d(e^iy_0, e^iz_0) \\
&\leq d(g^iz_0, z_i) + 3L + d(y_0, z_0) \\
&\leq d(g^iz_0, z_i) + 4L. 
\end{split}
\]
On the other hand, since $z_i$ lies on $[z_0, g^i z_0]$, 
we have $d(g^iz_0, z_i) = |d(z_0, g^iz_0) -d(z_0, z_i)|$. 
First suppose that $d(z_0, g^iz_0) \geq  d(z_0, z_i)$. Then 
\[ \begin{split}
d(g^iz_0, z_i)
&= d(z_0, g^iz_0) - d(z_0, z_i) \\
&\leq d(z_0, g^iz_0) - d(y_0, e^i y_0) + d(y_0, z_0) + d(e^i y_0, z_i) \\
&\leq d(z_0, g^iz_0) - d(x_0, g^i x_0) + 4L \\
&\leq i\tau(g) + K - i\tau(g) + 4L \\
&= K + 4L, 
\end{split}
\]
where we used \cref{lem:TandK} in the last inequality. 
The case $d(g^iz_0, z_i) < d(z_0, z_i)$ is treated similarly, and 
we again obtain $d(g^iz_0, z_i) \leq K + 4L$. Thus 
\[d(g^iz_0, e^i z_0) \leq d(g^iz_0, z_i) + 4L \leq 8L + K \leq 10 L + K,  \]
which shows $d(z_0, e^{-i}g^i z_0) \leq 10L + K$. 
Let $z_i'$ be a nearest-point projection of $e^iy_0'$ onto $l_g$.  
By similar computations, noting that $d(y_0', z_0') \leq 2L$, we obtain 
\[ \begin{split}
d(g^iz_0', e^iz_0') \leq d(g^iz_0', z_i') + 5L \leq 10 L + K, 
\end{split}
\]
and $d(z_0', e^{-i}g^i z_0') \leq 10L + K$ for any $i\in \{0,1,\ldots, N\}$. 

Then, by the definition of $R$ and $N$, there exist $i_1, i_2\in \{0,1,\ldots, N\}$
with $i_1\neq i_2$ such that $e^{-i_1}g^{i_1} = e^{-i_2}g^{i_2}$. 
This implies that $hg^{i_1-i_2}h^{-1} = g^{-(i_1-i_2)}$, and therefore, $g$ is achiral.
\end{proof}

\begin{proof}[Proof of \cref{th:uniform-lower-bd}]
Let $R>0$ and $N\in\ZZ_{>0}$ be constants as in \cref{prop:key-for-uniform-lower-bd}. 
Recall that 
\[ \tau_0 \coloneq \inf\{ \tau(g)\mid \text{$g$ is loxodromic} \} \]
is positive (\cref{th:t_0-is-positive}).  
Let $M_1 \in \ZZ_{>0}$ be an integer such that 
\[ M_1\tau_0 \geq \max\{T,2K+2,R+3L\}. \] 
Let $g\in \Gamma$ be a chiral loxodromic element. Set $g_1 \coloneq g^{M_1}$, and 
let $x_0\in S$ be a point such that $d(x_0, g_1x_0) - \inf_{z\in S}d(z, g_1z) \leq \delta$. 
Then 
$\tau(g_1) \geq {M_1} \tau_0 \geq \max\{T,2K+2,R+3L\}$, and it follows from 
\cref{prop:key-for-uniform-lower-bd} that 
there is no copy of $[x_0, g_1^{-(N+1)} x_0]$ 
contained in $V(l_{g_1}, L)$
whose orientation agrees with that of $l_{g_1}$. 

Put $g_2 \coloneq g_1^{N+1} = g^{M_1(N+1)}$. Let $n\in \ZZ_{>0}$, 
and let $\beta$ be an almost realizing path of 
$c_{[x_0, g_2^{-1} x_0], K+1}([x_0, g_1^{n} x_0])$.  
Then $\beta$ is contained in $V(l_{g_1}, L)$ by \cref{Fuji98_Lemma3.3}, 
and thus$|\beta|_{[x_0, g_2^{-1} x_0]} = 0$. 
This shows that
\[ c_{[x_0, g_2^{-1} x_0], K+1}([x_0, g_1^{n} x_0]) = 0 \]
for any $n \in \ZZ_{>0}$. 

Let $k \in \ZZ_{>0}$ and  
$\gamma = \bigcup_{0\leq i \leq k-1}g_2^i[x_0, g_2 x_0]$. 
Then, by \cref{lem:TandK}, we have 
\[ \begin{split}
|\gamma| - (K+1)|\gamma|_{[x_0, g_2 x_0]} 
&= k d(x_0, g_2 x_0) - (K+1)k  \\
&\leq k(\tau(g_2) + K)- (K+1)k \\
&= k\tau(g_2) - k, 
\end{split}
\]
and 
\[ \begin{split}
c_{[x_0, g_2 x_0]}([x_0, g_2^k x_0])
&\geq d(x_0, g_2^k x_0) - (|\gamma| - W|\gamma|_{[x_0, g_2 x_0]}) \\
&\geq k\tau(g_2) - \left(k\tau(g_2) - k\right) \\
&= k.
\end{split}
\]
Hence 
\[ h_{g_2}(g_2^k) 
= c_{[x_0, g_2 x_0], K+1}([x_0, g_2^k x_0]) 
- c_{[x_0, g_2^{-1} x_0], K+1}([x_0, g_1^{(N+1)k} x_0])
\geq k, 
\]
which means that $\overline{h_{g_2}}(g_2) \geq 1$. 
Therefore, letting $M \coloneq M_1(N+1)$, we obtain 
\[ \overline{h_{g^M}}(g)  
=\frac{1}{M}\overline{h_{g^M}}(g^M) 
=\frac{1}{M}\overline{h_{g_2}}(g_2) \geq \frac{1}{M}   
\]
as desired. 
\end{proof}

\subsection{General statements}

The following is a general statement of the gap theorems: 

\begin{thm}\label{Gap_thm: general}
Let $\Gamma$ be a group that admits 
an acylindrical action on a hyperbolic space, 
$M\in\ZZ_{>0}$ the constant as in \cref{th:uniform-lower-bd}, 
and  $A\subset\Gamma$ a symmetric subset containing every element of finite order. 

Suppose that there exists a positive constant $C$ such that 
$\sup\{ |\overline{h_{g^M}}(a)| \mid a\in A \} \leq C$
for any chiral loxodromic element $g\in\Gamma$. 

\begin{enumerate}
\item
A loxodromic element $g\in \Gamma$ is achiral if and only if $\overline{\nu_A}(g)=0$. 
\item 
There exists a positive constant $\Theta$ 
such that $\overline{\nu_A}(g)\ge \Theta$ for any chiral loxodromic element $g\in\Gamma$. 
\end{enumerate}
\end{thm}

\begin{proof}
Let $D$ (resp. $M$) be the constant as in \cref{prop:upper-bdd-of-defect} 
(resp. \cref{th:uniform-lower-bd}), 
and set $\Theta\coloneqq\frac{1}{M(D+C)}>0$. 
By \cref{achiral-WPD-normzero}, it suffices to show that, for any chiral loxodromic element $g\in\Gamma$, 
we have $\overline{\nu_A}(g)\ge\Theta$. 

By the above arguments, we have the followings:
\begin{itemize}
\item
$D_{\overline{h_{g^M}}}\le D$ (\cref{prop:upper-bdd-of-defect}), 
\item
$\overline{h_{g^M}}(g)\ge\frac{1}{M}$ (\cref{th:uniform-lower-bd}), 
\item
$\sup\{ |\overline{h_{g^M}}(a)| \mid a\in A \} \leq C$ (Assumption). 
\end{itemize}
Therefore, by \cref{basic-ineq}, we have
\[
\overline{\nu_A}(g)
\ge\frac{|\overline{h_{g^M}}(g)|}{D_{\phi_g}+C}
\ge\frac{1}{M(D+C)}=\Theta, 
\]
which completes the proof. 
\end{proof}

For examples of acylindrically hyperbolic groups, see \cite[\S 2.3]{OsinICM}.
The automorphism groups of certain algebraic varieties considered below are geometrically finite Kleinian groups, and hence non-elementary relatively hyperbolic groups with proper peripheral subgroups; these form a basic class of examples of acylindrically hyperbolic groups. 

\subsubsection{scl}

Let $\Gamma$ be a group and let $C(\Gamma)$ denote the symmetric subset consisting of all commutators in $\Gamma$. 
The associated norm $\nu_{C(\Gamma)}$ is called the \emph{commutator norm} (or \emph{commutator length}), 
and we adopt the standard notation $\scl:=\scl_\Gamma:=\overline{\nu_{C(\Gamma)}}$ for its stabilization. 
The following is the gap theorem for $\scl$. 

\begin{thm}\label{Gap_thm: scl}
Let $\Gamma$ be a group that admits 
an acylindrical action on a $\delta$-hyperbolic space. 
\begin{enumerate}
\item
A loxodromic element $g\in \Gamma$ is achiral if and only if $\scl(g)=0$. 
\item 
There exists a positive constant $\Theta$ 
such that $\scl(g)\ge \Theta$ for any chiral loxodromic element $g\in\Gamma$. 
\end{enumerate}
\end{thm}
\begin{proof}
Note that, since the subset $C(\Gamma)$ 
does not, in general, contain all finite order elements, 
we cannot apply \cref{achiral-WPD-normzero}. 
But in this case, direct computations imply $\scl(a)=0$ for any achiral element $a\in\Gamma$. 

The assertion follows from the proof of \cref{Gap_thm: general} and 
$|\overline{h_{g^M}}(a)|\le D_{\overline{h_{g^M}}}\le D$ for any $a\in C(\Gamma)$. 
\end{proof}

\begin{rem}
\begin{enumerate}
\item
Originally, the seminal work of Calegari--Fujiwara \cite{Cal-Fuj10} established the gap theorem for $\scl$ for hyperbolic groups and for pseudo-Anosov elements of mapping class groups.
Furthermore, in \cite[Theorems A and B]{BBF}, 
Bestvina--Bromberg--Fujiwara proved, as a generalization of \cite{Cal-Fuj10}, 
the gap theorem for $\scl$ for arbitrary elements of mapping class groups.

\item
We also note that if $\Gamma$ satisfies the gap theorem for $\scl_\Gamma$, 
then every subgroup $H \subset \Gamma$ also satisfies the gap theorem for $\scl_H$.
For example, geometrically infinite subgroups of (non-elementary) geometrically finite Kleinian groups satisfy the gap theorem for $\scl$.
\end{enumerate}
\end{rem}


\subsubsection{Translation-length norms}

We show the gap theorem for stable norms associated to the translation length. 

Let $\Gamma$ be a group admitting 
an acylindrical action on a $\delta$-hyperbolic space $S$. 
For any $l\in\RR_{\ge0}$, 
we define a symmetric subset 
\[
T_l(\Gamma):=\{g\in\Gamma\ \mid \tau_S(g)\le l \}
\]
of $\Gamma$, 
and set $\nu_l\coloneqq\nu_{T_l(\Gamma)}$ which is called the \emph{l-translation-length norm} in this paper.

\begin{thm}\label{Gap_thm: trans_length}
Let $l\in\RR_{\ge0}$. 
\begin{enumerate}
\item
A loxodromic element $g\in \Gamma$ is achiral if and only if $\overline{\nu_l}(g)=0$. 
\item
There exists a positive constant $\Theta$ 
such that $\overline{\nu_l}(g)\ge \Theta$ for any chiral loxodromic element $g\in\Gamma$. 
\end{enumerate}
\end{thm}
\begin{proof}
The assertion follows from \cref{Gap_thm: general} and the following lemma. 
\end{proof}

\begin{lem}\label{lem:uniform-bd}
For any $l\in\RR_{\ge0}$, a path $w$ 
of finite length,
and a real number $W$ with $0<W<|w|$, 
we have 
\[
\sup\{|\overline{h_{w,W,x_0}}(a)|~|~a\in T_l(\Gamma)\}\le l. 
\]
\end{lem}
\begin{proof}
The assertion follows from the computation
\begin{eqnarray*}
|\overline{h_{w,W,x_0}}(a)|
&=&\lim_{n\to \infty}\left|\frac{h_{w,W,x_0}(a^n)}{n}\right|\\
&\le&\lim_{n\to \infty}\frac{\max\{c_{w,W}([x_0,a^nx_0]),c_{w^{-1},W}([x_0,a^nx_0])\}}{n}\\
&\le&\lim_{n\to \infty}\frac{d(x_0,a^n x_0)}{n}
=\tau(a)\le l.
\end{eqnarray*}
\end{proof}

We apply this theorem for gap theorems for entropy norms in the following two subsections. 

\begin{rem}
In fact, the two gap theorems, Theorems \ref{Gap_thm: general} and \ref{Gap_thm: scl}, can also be proved by using the Bestvina--Bromberg--Fujiwara quasimorphisms constructed from actions on quasi-trees \cite{BBF}, in the same way as in our proofs. 
Uniform estimates for these quasimorphisms, corresponding to our \cref{prop:upper-bdd-of-defect} 
and \cref{th:uniform-lower-bd}, can be found in \cite[Corollary 3.2]{BBF}.

On the other hand, we note that the gap theorem for the translation-length norm (\cref{Gap_thm: trans_length}) was proved by using Epstein--Fujiwara quasimorphisms,  
because we required their property established in \cref{lem:uniform-bd}.  
\end{rem}

\subsection{Entropy norms on mapping class groups}
\label{subsection: gap-ent-MCG}

Let $\Sigma$ be a closed orientable surface of genus at least two, and let
$\MCG(\Sigma)$ denote its mapping class group.
By the Nielsen--Thurston classification, every mapping class is either
periodic, reducible, or pseudo-Anosov.
A mapping class is \emph{periodic} if it has finite order, and \emph{reducible} if it
preserves the isotopy class of an essential multicurve.
A mapping class that is neither periodic nor reducible is called
\emph{pseudo-Anosov}; such mapping classes exhibit rich dynamical behavior
described in terms of invariant measured foliations.

For a mapping class $f\in \MCG(\Sigma)$, its topological entropy is defined as
the infimum of the topological entropies of its representatives.
We note that positive topological entropy is not equivalent to being
pseudo-Anosov.
More precisely, a mapping class has positive topological entropy if and only if
it is either pseudo-Anosov or reducible with at least one pseudo-Anosov
component in its Nielsen--Thurston decomposition.
For background on the Nielsen--Thurston classification and pseudo-Anosov
theory, see \cite[Chapters~13 and 14]{PrimerMCG}; for the entropy statements,
see \cite[Expos\'e 1, Theorems 1.6 and 1.8]{FLP-English}.

For any $l\in\RR_{\ge0}$,
we define the symmetric subset
\[
E_l(\Sigma):=\{g\in\MCG(\Sigma)\mid h_{\mathrm{top}}(g)\le l\}
\]
of $\MCG(\Sigma)$,
and set $e_l\coloneqq \nu_{E_l(\Sigma)}$, which we call the \emph{$l$-entropy norm}.


\begin{rem}
The entropy norm was introduced by Brandenbursky--Marcinkowski in \cite{Brandenbursky-Marcinkowski2019} and has since been studied in \cite{Brandenbursky-Kabiraj2020} and \cite{Brandenbursky-Shelukhin2021}.
Their entropy norm is precisely the special case $l=0$ of ours.
\end{rem}

The curve graph $\C(\Sigma)$ is $\delta$-hyperbolic, where $\delta$ depends only on $\Sigma$, 
and an element of $\MCG(\Sigma)$ acts loxodromically on $\C(\Sigma)$ if and only if it is pseudo-Anosov \cite{Masur-Minsky}.
Furthermore, the action of $\MCG(\Sigma)$ on $\C(\Sigma)$ is acylindrical \cite{Bowditch2008}, see also \cite[Lemma 6.49]{DGO17}.

\cref{Gap_thm: trans_length} yields the gap theorem for the stable entropy norm:

\begin{thm}\label{Gap_thm: ent_MCG}
Let $\Sigma$ be a closed orientable surface of genus at least two,
and let $l\in\RR_{\ge0}$. 
\begin{enumerate}
\item
An pseudo-Anosov element $g\in\MCG(\Sigma)$ is achiral if and only if $\overline{e_l}(g)=0$. 
\item
There exists a positive constant $\Theta$ 
depending only on the genus of $\Sigma$ and $l$,
such that $\overline{e_l}(g)\ge \Theta$ for any chiral pseudo-Anosov element $g\in\MCG(\Sigma)$. 
\end{enumerate}
\end{thm}


\begin{proof}
Since every finite-order (i.e. periodic) element has zero topological entropy,
\cref{achiral-WPD-normzero} implies that $\overline{e_l}(g)=0$
for every achiral pseudo-Anosov element $g\in\MCG(\Sigma)$.

Consider the action of $\MCG(\Sigma)$ on the curve graph.
Since every periodic or reducible element has zero translation length,
we have $E_0(\Sigma)\subset T_0(\MCG(\Sigma))$.
Since there are only finitely many conjugacy classes of pseudo-Anosov elements
$g$ with $h_{\mathrm{top}}(g)\le l$ (\cite{Iva90}),
we have
\[
E_l(\Sigma)\setminus E_0(\Sigma)\subset T_{l'}(\MCG(\Sigma))
\]
for some $l'\in\RR_{\ge0}$ whenever $l>0$.
Thus, we obtain
$
\overline{e_l}(h)\ge \overline{\nu_{l'}}(h)
$
for any $h\in\MCG(\Sigma)$.
The assertion then follows from the gap theorem for $\overline{\nu_{l'}}$
(\cref{Gap_thm: trans_length}).
\end{proof}

\subsection{Entropy norms on automorphism groups}
\label{subsection: gap-ent-algvar}

Let $X$ be a K3 surface or an Enriques surface over an algebraically closed field, or a projective IHS manifold. 

Recall that we defined the entropy of an automorphism of $X$ and 
observed the natural action of $\Aut(X)$ on the associated hyperbolic space $\HH_X$ in 
\S\ref{ss:surfaces} and \S\ref{ss:IHS}. 
For any $l\in\RR_{\ge0}$, the \emph{$l$-entropy norm} $e_l$ on $\Aut(X)$ is defined in the same way as for mapping class groups (see \S\ref{subsection: gap-ent-MCG}).
In general, the action of $\Aut(X)$ on $\HH_X$ is not acylindrical,
since $\Aut(X)$ may contain parabolic elements.
Nevertheless, \cref{Gap_thm: trans_length} yields the gap theorem for the stable entropy norm
when combined with the geometrical finiteness of $\Aut(X)$ and 
the acylindrical action of a geometrically finite Kleinian group on a certain hyperbolic graph 
constructed in \cite{Osin}. 

\begin{thm}\label{Gap_thm: ent_Aut}
Let $X$ be a K3 surface or an Enriques surface over an algebraically closed field of characteristic different from $2$, or a projective irreducible holomorphic symplectic manifold, and let $l\in\RR_{\ge0}$. 
\begin{enumerate}
\item
An automorphism $g\in \Aut(X)$ of positive entropy is achiral if and only if $\overline{e_l}(g)=0$. 
\item
There exists a positive constant $\Theta$ 
such that $\overline{e_l}(g)\ge \Theta$ for any chiral automorphism $g\in\Aut(X)$
of positive entropy.  
\end{enumerate}
\end{thm}



\begin{proof}
For the action of $\Aut(X)$ on the hyperboloid $\HH_X$,
we note that the $l$-entropy norm coincides with the associated 
$\frac{2l}{\dim X}$-translation-length norm by \cref{lem:tl_equals_ent} or \cref{lem:tl_and_ent:HK}. 
Furthermore, since this action is proper, we have $\overline{e_l}(a)=0$ for any achiral element $a\in\Aut(X)$ of positive entropy by \cref{achiral-WPD-normzero}.

Since $\Aut(X)$ acts on $\HH_X$ as a geometrically finite Kleinian group by \cite[Theorem A]{Kikuta} (also by \cite{Takatsu} for complex K3 surfaces),
$\Aut(X)$ is hyperbolic relative to a finite collection $\{G_1,\dots,G_k\}$ of maximal parabolic subgroups \cite{Bowditch-rel-hyp} (see also \cite{Osin-rel-hyp}).
Then every parabolic element is conjugate into some $G_i$ \cite[Proof of Theorem 12.4.3]{Rat}. 

By \cite[Proposition 4.28]{DGO17} and \cite[Proposition 5.2]{Osin}, 
there exists a finite symmetric subset $A$ of $\Aut(X)$ such that $A\sqcup \bigcup_i G_i$ generates $\Aut(X)$ and 
the Cayley graph
\[
\C\!\left(\Aut(X),\, A\sqcup\bigcup_i G_i\right)
\]
is hyperbolic, 
and the action of $\Aut(X)$ on this graph is acylindrical.


Note that every automorphism of zero entropy has translation-length zero with respect to the action on this Cayley graph; 
hence we have 
$E_0(X)\subset T_0(\Aut(X))$.
Since the number of conjugacy classes of automorphisms $g$ 
with positive entropy at most $l$ is finite \cite[Theorem 12.7.8]{Rat},
we have
\[
E_l(X)\setminus E_0(X)\subset T_{l'}(\Aut(X))
\]
for some $l'\in\RR_{\ge0}$ whenever $l>0$.
Thus, we obtain
$
\overline{e_l}(h)\ge \overline{\nu_{l'}}(h)
$
for any $h\in\Aut(X)$.
The assertion now follows from the gap theorem for $\overline{\nu_{l'}}$
(\cref{Gap_thm: trans_length}), 
since every automorphism of positive entropy acts loxodromically on the Cayley graph by \cite[Corollary 4.20]{Osin-rel-hyp}.
\end{proof}



\section{Achirality of parabolic automorphisms}\label{sec:para}

\Cref{Gap_thm: ent_Aut} gives a criterion for the achirality of loxodromic automorphisms 
of varieties under consideration.  
Note that every elliptic automorphisms of such a variety is achiral since its has finite order.  
In this section, we study the achirality of parabolic automorphisms of K3 surfaces and Enriques surfaces. Recall that an automorphism of a K3 surface or an Enriques surface is parabolic 
if and only if it is of infinite order and preserves a genus-one fibration (\cref{prop:trichotomy}).
We will make use of the Jacobian fibration associated with a genus-one fibration.

\subsection{Genus-one curves}\label{ss;Genus-oneCurves}

We start with genus-one curves. See \cite{Kleiman_PicardScheme} 
for Picard schemes. 
Let $C$ be a geometrically integral projective curve of genus one over a field $K$. 
The open subscheme $C_0\subset C$ where $C_0/K$ is smooth is naturally isomorphic to 
the degree-one component $\bPic_{C/K}^1$ of the Picard scheme $\bPic_{C/K}$, 
which induces the bijection $r \mapsto \O_{C_L}(r)$ on 
the $L$-valued points for any field $L$ over $K$.  
Moreover, $C_0$ is naturally an $E$-torsor, where
$E\coloneq \bPic_{C/K}^0$ is the Jacobian variety of $C/K$.  
Let $L$ be a field over $K$ such that $C_0(L)$ is non-empty 
(e.g. the separable closure or the algebraic closure). 
Then, $C_0(L)$ is an $E(L)$-torsor. 
The action of $E(L)$ on $C_0(L)$ is written as 
$r + s$ for $r\in C_0(L)$ and $s\in E(L)$.
We remark that 
$r - r' = \O_{C_{L}}(r-r')$ in $E(L) = \Pic^0(C_{L})$
for any two points $r,r'\in C_0(L)$, 
where the left-hand side is the unique point $s\in E(L)$ such that
$r = r' + s$, and the right-hand side is the line bundle associated with the 
divisor $r - r'$ on $C_{L}$. 

We now assume that $C$ is regular.  
Then, every automorphism of $C_0$ uniquely extends to an automorphism of $C$. 
Thus, the $E$-torsor structure on $C_0$ gives rise to a group homomorphism 
\[ \tau: E(K) \to \Aut_{K}(C) \] 
defined by 
$\tau_s(r) = r + s$ for $s\in E(K)$ and $r\in C_0(K)$, 
where we write $\tau_s$ instead of $\tau(s)$.  
Let $\Aut_{\mathrm{gp}}(E)$ be the group of automorphisms of 
$E$ as a group scheme over $K$, and let 
$\varphi:\Aut_{K}(C)\to \Aut_{\mathrm{gp}}(E)$
be the group homomorphism defined by the push-forward of line bundles 
(we use the push-forward instead of the pull-back to avoid $\varphi$ being 
an anti-homomorphism). 

Let $p_{C}: C\to \Spec(K)$ be the structure morphism. 
We define $\Aut_{p_C}(C)$ to be the 
group of automorphisms $g:C\to C$ as a scheme (over $\ZZ$)
for which there exists an automorphism 
$b(g):\Spec(K)\to \Spec(K)$ 
such that $b(g)\circ p_{C} = p_{C} \circ g$. 
The automorphism $b(g)$ is uniquely determined by $g$, and we have 
\[ \Aut_{K}(C) 
= \{g\in \Aut_{p_{C}}(C) \mid b(g) = \id_{\Spec(K)} \}. 
\]
We now extend the group homomorphism $\varphi$
to a homomorphism $\Aut_{p_{C}}(C) \to \Aut_{p_E}(E)$ 
where $p_E:E\to \Spec(K)$ is the structure morphism. 
Let $g\in \Aut_{p_{C}}(C)$. Then, 
$g$ can be seen as a $K$-isomorphism from 
$C\xrightarrow{p_{C}} \Spec(K) \xrightarrow{b(g)}\Spec(K)$ 
to 
$C\xrightarrow{p_{C}} \Spec(K)$, 
and the push-forward of line bundles defines a $K$-isomorphism $\varphi(g)$ from 
$E\xrightarrow{p_{E}} \Spec(K) \xrightarrow{b(g)}\Spec(K)$ 
to 
$E\xrightarrow{p_{E}} \Spec(K)$.
Hence, we obtain a homomorphism 
\[\varphi:\Aut_{p_{C}}(C) \to \Aut_{p_{E}}(E), \] 
which extends the homomorphism $\varphi:\Aut_{K}(C)\to \Aut_{\mathrm{gp}}(E)$
defined above. 
Indeed, we have $b(\varphi(g)) = b(g)$ for $g\in \Aut_{p_C}(C)$ by construction. 
Let $g\in \Aut_{p_C}(C)$. For a $\overline{K}$-point $r\in C(\overline{K})$,  
the composition $g \circ r \circ b(g)^{-1}$ is again a $\overline{K}$-point of $C$. 
Moreover, we have the formula 
\[ g\circ (r + s) \circ b(g)^{-1} 
= (g\circ r \circ b(g)^{-1}) + (\varphi(g)\circ s \circ b(g)^{-1})  \]
where $r\in C(\overline{K})$ and $s\in E(\overline{K})$.   


\begin{prop}\label{prop:conjugation-of-translation}
Let $s\in E(K)$ and $g\in \Aut_{p_C}(C)$. We have 
$g\circ \tau_s \circ g^{-1} = \tau_{\varphi(g)\circ s \circ b(g)^{-1}}$
(as automorphisms of $C$). 
\end{prop}
\begin{proof}
For $r\in C(\overline{K})$, we have
\[\begin{split}
g\circ \tau_s \circ g^{-1}\circ r 
&= g\circ \tau_s \circ g^{-1}\circ r \circ b(g) \circ b(g)^{-1} \\
&= g\circ (g^{-1}\circ r \circ b(g) + s) \circ b(g)^{-1} \\
&= r + \varphi(g)\circ s \circ b(g)^{-1} \\
&= \tau_{\varphi(g)\circ s \circ b(g)^{-1}}(r). 
\end{split}
\]
This means that we obtain the desired equation after base change to 
$\overline{K}$ since $C_{\overline{K}}$ is an variety. 
This implies that 
$g\circ \tau_s \circ g^{-1} = \tau_{\varphi(g)\circ s \circ b(g)^{-1}}$
since $\Spec(\overline{K})\to \Spec(K)$ is a flat morphism. 
\end{proof}

Let $[-1]\in \Aut_{\mathrm{gp}}(E)$ denote the automorphism of $E$ 
given by multiplication by $-1$. 

\begin{cor}\label{cor:[-1]_implies_achiral}
Let $s \in E(K)$. Suppose that there exists $g\in \Aut_K(C)$
such that $\varphi(g) = [-1]$. 
Then $g\circ \tau_s\circ g^{-1} = \tau_{-s} = \tau_s^{-1}$. 
In particular, $\tau_s$ is achiral. \qed
\end{cor}

\subsection{Genus-one fibrations}\label{ss:Genus-oneFibrations}
Let $k$ be an algebraically closed field. 
Let $p:X\to B$ be a genus-one fibration from a smooth projective surface $X$ to 
a smooth projective curve $B$ over $k$. 
The group of automorphisms of $X$ that preserve $p$ is denoted by $\Aut_p(X)$. 
Let $K$ be the function field of $B$, and
let $E \xrightarrow{p_E} K$ be the Jacobian variety of the genus-one curve $X_K$. 
Then, there exists a genus-one fibration $j:J\to B$ that admits a section and 
satisfies the following 
properties (\cite[\S\S 4.2, 4.3]{Cossec-Dolgachev-Liedtke_EnriquesI}): 

\begin{enumerate}
\item $J$ is a smooth projective surface over $k$, and $j:J\to B$ is relatively minimal, 
that is, the relative canonical sheaf $\omega_{J/B}$ is nef. 
\item $j|_{J^\sharp}:J^\sharp\to B$ is a N\'eron model of $E$, 
where $J^\sharp \subset J$ is the open subscheme where $j|_{J^\sharp}:J^\sharp\to B$ is smooth. 
Namely, $((J^\sharp)_K \xrightarrow{j|_{(J^\sharp)_K}} K) = (E\xrightarrow{p_E} K)$, 
and for any smooth $B$-scheme $Y$ and any $K$-morphism $u_K: Y_K \to (J^\sharp)_K$, there exists 
a unique $B$-morphism $u: Y \to J^\sharp$ extending $u_K$. 
\end{enumerate}
This fibration $j$ is determined uniquely by $p$. We refer to $j$ as the 
\emph{Jacobian fibration associated to $p$} and write $J_p: J_X\to B$. 

There is a natural group homomorphism 
$\Aut_{J_p}(J_X) \to \Aut_{p_E}(E)$ given by the restriction to the generic fiber. 
This is an isomorphism. Indeed, the inverse is constructed as follows. 
Let $h\in \Aut_{p_E}(E)$. Note that $b(h):\Spec(K)\to \Spec(K)$ uniquely extends to 
an automorphism of $B$ since $B$ is regular projective curve; 
we denote this automorphism again by $b(h):B\to B$. 
On the other hand, $h$ extends to an birational map $\tilde{h}:J_X\to J_X$,  
and the diagram
\[ \xymatrix{
J_X \ar@{.>}[r]^{\tilde{h}} \ar[d]^{J_p} & J_X \ar[d]^{J_p}\\
B \ar[r]^{b(h)} & B 
} \]
commutes. 
Since $J_X\to B$ is relatively minimal, the birational map $\tilde{h}$ is in fact 
an automorphism. 
Then, the group homomorphism $\Aut_{p_E}(E) \to \Aut_{J_p}(J_X),\; h\mapsto \tilde{h}$
is the inverse of the natural homomorphism $\Aut_{J_p}(J_X) \to \Aut_{p_E}(E)$. 

We now obtain a group homomorphism 
\begin{equation*}
\Aut_{p}(X)
\xrightarrow{} \Aut_{p|_{X_K}}(X_K)
\xrightarrow{\varphi} \Aut_{p_E}(E)
\xrightarrow{\cong} \Aut_{J_p}(J_X),  
\end{equation*}
where the first arrow is the homomorphism given by the restriction to the generic fiber, 
the second arrow is the homomorphism constructed in the previous subsection, 
and third arrow is the natural isomorphism as discussed above.
By the abuse of notation, this homomorphism is also denoted by 
$\varphi: \Aut_{p}(X)\to \Aut_{J_p}(J_X)$. 

Recall that $\Aut_{p_E}(E)$ contains the automorphism $[-1]$. 
We use the same symbol $[-1]$ to denote the corresponding element in 
$\Aut_{J_p}(J_X)$ under the natural isomorphism 
$\Aut_{J_p}(J_X) \cong \Aut_{p_E}(E)$. 


Assume that the fibration $p:X\to B$ is nontrivial, i.e., not isomorphic to a product. 
Then, the Mordell--Weil Theorem (\cite[Theorem 4.3.3]{Cossec-Dolgachev-Liedtke_EnriquesI}) 
states that $E(K)$ is a finitely generated group. 
The group $E(K)$ is called the \emph{Mordell--Weil group} of $p$, and 
the rank of $E(K)$ is called the \emph{Mordell--Weil rank} of $p$. 
It is known as the \emph{Shioda--Tate formula} that the Mordell--Weil rank 
is given by 
\[ \rho(X) - 2 - \sum_{t\in B}(\#\mathrm{Irr}(X_t) - 1) \]
where $\#\mathrm{Irr}(X_t)$ is the number of the irreducible components of the fiber $X_t$; see
\cite[Formula (4.3.1), Corollary 4.3.18, and Theorem 4.3.20]{Cossec-Dolgachev-Liedtke_EnriquesI}. 

Assume that $p:X\to B$ is relatively minimal. 
Then, there exists a natural isomorphism $\Aut_{p|_{X_K}}(X_K) \cong \Aut_p(X)$. 
Furthermore, we have the injective homomorphism $\tau:E(K)\to \Aut_{p|_{X_K}}(X_K)$. 
Hence, the Mordell--Weil group $E(K)$ can be regarded as a subgroup of $\Aut_p(X)$. 
The image of $\tau:E(K)\to \Aut_{p|_{X_K}}(X_K) \cong \Aut_p(X)$ is denoted by 
$\MW(p) \subset \Aut_p(X)$. \cref{cor:[-1]_implies_achiral} yields:

\begin{lem}\label{lem:[-1]_implies_achiral:fib}
Suppose that there exists $h\in \Aut_p(X)$
such that $\varphi(h) = [-1]$. 
Then $h\circ g \circ h^{-1} = g^{-1}$ for any $g\in \MW(p)$. 
In particular, $\MW(p)$ is uniformly achiral. \qed
\end{lem}

We remark that every genus-one fibration $X\to B$ is automatically nontrivial and 
relatively minimal when $X$ is a K3 surface or an Enriques surface. 

\begin{prop}\label{prop:finite_index}
Suppose that $X$ is a K3 surface or an Enriques surface. 
Then $\MW(p)$ is a finite index subgroup of $\Aut_p(X)$. 
\end{prop}
\begin{proof}
Note that $\MW(p)$ is a finitely generated abelian group of rank 
$r \coloneq \rho(X) - 2 - \sum_{t\in B}(\#\mathrm{Irr}(X_t) - 1)$.  
Since $\varrho:\Aut(X)\to \mathrm{O}(\Num(X))$ 
is of finite kernel (\cref{prop:finite_kernel}), 
$\varrho(\MW(p))$ is an abelian group of rank $r$, and it suffices to show that 
$\varrho(\Aut_p(X))$ is a virtually abelian group of rank $\leq r$. 

Let $e\in \Num(X)$ be a primitive vector 
lying on the line spanned by the fiber class of the fibration $p$. Then 
\[\varrho(\Aut_p(X))\subset \mathrm{O}(\Num(X))_e 
\coloneq \{ t\in \mathrm{O}(\Num(X)) \mid t(e) = e \}, \] 
and thus,  
$\varrho(\Aut_p(X))$ naturally acts on the lattice $e^\perp/\ZZ e$. 
On the other hand, the classes of $(-2)$-curves in fibers over $p$ spans 
a root lattice $R$ in $e^\perp/\ZZ e$, which is preserved by $\varrho(\Aut_p(X))$. 
Hence, by a similar argument to \cite[Proposition 3.2]{Brandhorst-Mezzedimi2022}, 
it follows that $\varrho(\Aut_p(X))$ is virtually abelian and 
\[ \rk \varrho(\Aut_p(X)) \leq \rho(X) - 2 - \rk R = r. \]
This completes the proof. 
\end{proof}

\begin{prop}\label{prop:MWrank0}
Suppose that $X$ is a K3 surface or an Enriques surface. 
If the Mordell--Weil rank of $p$ is $0$, then 
$\Aut_p(X)$ is uniformly achiral. 
If $p$ is a quasi-elliptic fibration, then 
$\Aut_p(X)$ is uniformly achiral. 
\end{prop}
\begin{proof}
Suppose that the Mordell--Weil rank of $p$ is $0$. 
Then, $\MW(p)$ is a finite group, and so is $\Aut_p(X)$ by 
\cref{prop:finite_index}. Hence $\Aut_p(X)$ is uniformly achiral. 

Suppose that $p$ is a quasi-elliptic fibration. 
Then $p$ has Mordell--Weil rank $0$ by 
\cite[Theorem 4.3.3]{Cossec-Dolgachev-Liedtke_EnriquesI}. 
Therefore, $\Aut_p(X)$ is uniformly achiral by the former statement. 
\end{proof}

\subsection{The Weil--Ch\^atelet group}
We refer to \cite{Serre_GaloisCohomology} for Galois cohomology. 
Let $K$ be a field, $K_\mathrm{s}$ its separable closure, and 
$G \coloneq \Gal(K_\mathrm{s}/K)$. 
We first recall the natural action of $G$ on sets associated with schemes. 
For $\sigma\in G$, we write 
$\sigma^* = \Spec(\sigma):\Spec(K_\mathrm{s})\to \Spec(K_\mathrm{s})$. 
Let $X$ and $Y$ be $K$-schemes.  
The Galois group $G$ acts on the left on the set 
$\Hom_{K_\mathrm{s}}(X_{K_\mathrm{s}}, Y_{K_\mathrm{s}})$ of $K_\mathrm{s}$-morphisms 
from $X_{K_\mathrm{s}}$ to $Y_{K_\mathrm{s}}$ by 
${}^\sigma f = (\id_Y\times(\sigma^*)^{-1})\circ f \circ (\id_X\times\sigma^*)$ 
for $\sigma\in G$ and $f\in \Hom_{K_\mathrm{s}}(X_{K_\mathrm{s}}, Y_{K_\mathrm{s}})$. 
In particular, $G$ acts on the left on the set $X(K_\mathrm{s})$ of 
$K_\mathrm{s}$-points. 
One has $\Hom_{K}(X, Y) = \Hom_{K_\mathrm{s}}(X_{K_\mathrm{s}}, Y_{K_\mathrm{s}})^G$, 
and $X(K) = X(K_\mathrm{s})^G$. 
We write $f(x) = f\circ x \in Y(K_\mathrm{s})$ for $x\in X(K_\mathrm{s})$ and 
$f\in \Hom_{K_\mathrm{s}}(X_{K_\mathrm{s}}, Y_{K_\mathrm{s}})$. 
We have the following formula:
\[ (^\sigma f)(x) = {}^{\sigma}(f({}^{\sigma^{-1}}{x}))
\quad (x\in X(K_\mathrm{s}), f\in \Hom_{K_\mathrm{s}}(X_{K_\mathrm{s}}, Y_{K_\mathrm{s}}), 
\sigma\in G). 
\]
If $f\in \Hom_{K_\mathrm{s}}(X_{K_\mathrm{s}}, Y_{K_\mathrm{s}})$ is an isomorphism 
then so is ${}^\sigma f$. 
Hence, the automorphism group $\Aut_{K_\mathrm{s}}(X_{K_\mathrm{s}})$ is 
naturally regarded as a left $G$-group. 

Now, let $C$ be a smooth projective curve of genus one over $K$. 
Then, its Jacobian variety $E$ is an elliptic curve. 
We have the following exact sequence of $G$-groups:
\[
1 \xrightarrow{} E(K_\mathrm{s})
\xrightarrow{\tau} \Aut_{K_\mathrm{s}}(C_{K_\mathrm{s}})
\xrightarrow{\varphi} \Aut_{\mathrm{gp}}(E_{K_\mathrm{s}})
\xrightarrow{} 1
\]
where $\tau$ and $\varphi$ are the homomorphisms defined in \S\ref{ss;Genus-oneCurves}. 
By taking Galois cohomology, we obtain the exact sequence
\[
1 \xrightarrow{} E(K)
\xrightarrow{\tau} \Aut_{K}(C)
\xrightarrow{\varphi} \Aut_{\mathrm{gp}}(E)
\xrightarrow{\delta} H^1(G,E(K_\mathrm{s})). 
\]
The Galois cohomology group $H^1(G,E(K_\mathrm{s}))$ is called the 
\emph{Weil--Ch\^atelet group} of $E$ and is denoted by $\mathrm{WC}(E/K)$. 
Let $[C]\in \mathrm{WC}(E/K)$ denote the cohomology class of the $E$-torsor $C$.
The order of $[C]$ in $\mathrm{WC}(E/K)$ is referred to as the 
\emph{period} of the $E$-torsor $C$. 

\begin{lem}\label{lem:existenceof-1}
The image of $\varphi:\Aut_K(C)\to \Aut_{\mathrm{gp}}(E)$ contains $[-1]$ 
if and only if the period of the $E$-torsor $C$ is at most $2$. 
\end{lem}
\begin{proof}
Fix $p_0 \in C(K_\mathrm{s})$, and let $\theta:E_{K_\mathrm{s}}\to C_{K_\mathrm{s}}$
denote the isomorphism determined by
$E(K_\mathrm{s})\to C(K_\mathrm{s}), e \mapsto p_0 + e$. 
Note that $[C]$ is the class of the cocycle $({}^\sigma p_0 - p_0)_\sigma$. 

We compute $\delta([-1])$. 
Take $f\in \Aut_{K_\mathrm{s}}(C_{K_\mathrm{s}})$ such that $\varphi(f) = [-1]$. 
We may assume that $\theta^{-1}\circ f \circ \theta = [-1]$. 
For any $p\in C(K_\mathrm{s})$, we have
\[ f(p) 
= \theta\circ (\theta^{-1} \circ f \circ \theta )(\theta^{-1}(p)) 
= \theta([-1](p-p_0))
= p_0 + (p_0-p),
\]
and 
\[ \begin{split}
(f^{-1}\circ {}^\sigma f)(p)
&= f^{-1}({}^{\sigma}(f({}^{\sigma^{-1}} p)) \\
&= f^{-1}({}^{\sigma}(p_0 + (p_0- {}^{\sigma^{-1}} p))\\
%
%
&= p_0 + (p_0 - (^\sigma p_0 + ({}^\sigma p_0 - p)))\\
&=p - [2]({}^\sigma p_0 - p_0)\\
&=\tau_{-[2]({}^\sigma p_0 - p_0)}(p). 
\end{split}
\]
This means that $\delta([-1]) = -2[C]$.
Hence $\delta([-1]) = 0$ if and only if the order of $[C]$ is at most $2$, 
and the statement follows. 
\end{proof}

\begin{prop}\label{prop:exsistenceof-1}
Let $X$ be a smooth projective surface over an algebraically closed field, 
and let $p:X\to B$ be a relatively minimal elliptic fibration. 
The following conditions are equivalent: 
\begin{enumerate}
\item The image of $\varphi:\Aut_p(X)\to \Aut_{J_p}(J_X)$ contains $[-1]$. 
\item The multisection index of $p$ is either $1$ or $2$. 
\end{enumerate}
Moreover, if $X$ is a K3 surface or an Enriques surface, 
and if one of these conditions holds, then $\Aut_p(X)$ is uniformly achiral. 
\end{prop}
\begin{proof}
As in \S \ref{ss:Genus-oneFibrations}, let $K$ be the function field of $B$, and  
let $E$ be the Jacobian variety of the generic fiber $X_K$. 
We have the commutative diagram
\[ \xymatrix{
\Aut_p(X) \ar[r]^{\varphi} \ar[d]^{\cong} & \Aut_{J_p}(J_X) \ar[d]^{\cong} \\
\Aut_{p|_{X_K}}(X_K) \ar[r]^{\varphi} & \Aut_{p_E}(E) 
} \]
where each of the vertical arrows is given by the restriction to the generic fiber.
Hence, by \cref{lem:existenceof-1}, the condition (i) holds if and only if 
the period of the $E$-torsor $X_K$ is at most $2$. On the other hand, 
it is known that the period of $X_K$ coincides with the multisection index
of $p$; see \cite[Corollary 4.6.6]{Cossec-Dolgachev-Liedtke_EnriquesI}. 
Thus the conditions (i) and (ii) are equivalent. 

Suppose that $X$ is a K3 surface or an Enriques surface, and that  
there exists an automorphism $h\in \Aut_p(X)$ such that $\varphi(h) = [-1]$. 
Then, by \cref{lem:[-1]_implies_achiral:fib} and \cref{prop:finite_index}, 
it follows that $\Aut_p(X)$ is uniformly achiral. 
\end{proof}

\subsection{Achirality of parabolic automorphisms}

\begin{thm}\label{th:para-auto-of-Enriques}
Let $X$ be an Enriques surface over an algebraically closed field. 
Then $\Aut_p(X)$ is uniformly achiral for each genus-one fibration $p:X\to B$. 
In particular, every parabolic automorphism of $X$ is achiral.  
\end{thm}
\begin{proof}
Let $p:X\to B$ be a genus-one fibration. 
If $p$ is quasi-elliptic then $\Aut_p(X)$ is uniformly achiral by \cref{prop:MWrank0}. 
Suppose that $p$ is elliptic. 
It is known that every genus-one fibration of an Enriques surface has multisection index $2$ 
(see e.g. \cite[Propositions 1.5.1 and 2.2.8]{Cossec-Dolgachev-Liedtke_EnriquesI}). 
Hence, $\Aut_p(X)$ is uniformly achiral by \cref{prop:exsistenceof-1}. 
The latter assertion follows since every parabolic automorphism of $X$ is contained in 
$\Aut_p(X)$ for some genus-one fibration $p:X\to B$ by \cref{prop:trichotomy}.  
\end{proof}

The case of K3 surfaces is more complicated. 
In fact, we will see examples of elliptic K3 surfaces $p:X\to B$ in \S \ref{ss:Examples}, 
and one of them admits both achiral parabolic elements and chiral parabolic elements 
in $\Aut_p(X)$.  

We close this subsection by giving another sufficient condition for $\Aut_p(X)$ to be uniformly achiral. 
Let $X$ be a K3 surface over an algebraically closed field, 
and let $p:X \to B$ be an elliptic fibration. 
As in \S \ref{ss:Genus-oneFibrations}, let $K$ be the function field of $B$, and  
let $E$ be the Jacobian variety of the generic fiber $X_K$. 
The order of an automorphism $g$ is denoted by $\ord(g)$. 
We remark that the image of $\varphi:\Aut_p(X) \to \Aut_{J_p}(J_X)$ has 
finite order by \cref{prop:finite_index}.

\begin{lem}\label{lem: order-of-b-and-phi}
Let $h\in \Aut_p(X)$. 
\begin{enumerate}
\item $\ord b(h)$ divides $\ord \varphi(h)$. 
\item If there exists a parabolic automorphism $g\in \Aut_p(X)$ 
such that $h \circ g \circ h^{-1} = g^{-1}$, then $\ord \varphi(h)$ is even.
\end{enumerate}
\end{lem}
\begin{proof}
(i) follows from the equality $b(h) = b(\varphi(h))$.  
We prove (ii). Put $n \coloneq \ord\varphi(h)$, and 
let $g \in \Aut_p(X)$ be a parabolic automorphism such that $h \circ g \circ h^{-1} = g^{-1}$.  
By \cref{prop:finite_index}, after replacing $g$ by its positive power, 
we may assume that $g = \tau_s$ for some $s \in E(K)$. 
Then, we have $h \circ \tau_s \circ h^{-1} = \tau_{-s}$. 
By repeatedly conjugating by $h$, we obtain
\[  h^n \circ \tau_s \circ h^{-n} = \tau_{[-1]^{n} s}.  \]
On the other hand, by \cref{prop:conjugation-of-translation}, we have 
\[ h^n \circ \tau_s \circ h^{-n} 
= \tau_{\varphi(h)^{n} \circ s \circ b(h)^{-n}} 
= \tau_{s}.  \]
Hence $\tau_s = \tau_{[-1]^{n} s}$. 
Since $\tau_\bullet$ is injective, it follows that $s = [-1]^n s$. 
Note that $s$ is a non-torsion point since $g$ is of infinite order. 
Thus, we conclude that $\ord \varphi(h) = n$ is even.  
\end{proof}

\begin{thm}\label{th:para_K3}
Let $X$ be a K3 surface over an algebraically closed field,
and let $p: X \rightarrow B$ be an elliptic fibration of positive Mordell--Weil rank.
Then the following conditions are equivalent:
\begin{enumerate}
\item There exist $g, h \in \Aut_p(X)$ such that $g$ is parabolic, 
$\ord b(h) \neq \ord \varphi(h)$, and $h \circ g \circ h^{-1} = g^{-1}$. 
\item The image of $\varphi:\Aut_p(X)\to \Aut_{J_p}(J_X)$ contains $[-1]$. 
\end{enumerate}
Moreover, if one of these conditions holds, then $\Aut_p(X)$ is uniformly achiral. 
\end{thm}
\begin{proof}
We first show that (ii) implies (i). 
Since $p$ has positive Mordell--Weil rank, there exists a parabolic automorphism 
$g\in \MW(p)\subset \Aut_p(X)$. 
Let $h\in \Aut_p(X)$ be an automorphism such that $\varphi(h) = [-1]$. 
Then $\ord(\varphi(h)) = \ord([-1]) = 2$ and $\ord(b(h)) = \ord(b([-1])) = 1$.
Thus $\ord b(h) \neq \ord \varphi(h)$. 
Furthermore, we have $h \circ g \circ h^{-1} = g^{-1}$ by \cref{lem:[-1]_implies_achiral:fib}. 

We then show that (i) implies (ii).
Let $g, h \in \Aut_p(X)$ be as in the condition (i), and put 
$n \coloneq \ord \varphi(h)$ and $m \coloneq \ord b(h)$. 
We prove that $\varphi(h^m) = [-1]$. 
Since $m \leq n$ by \cref{lem: order-of-b-and-phi} (i) and $m\neq n$ by assumption, we have 
\[  \ord \varphi(h^m)=\frac{n}{m}>1. \]
As in the proof of \cref{lem: order-of-b-and-phi}~(ii), 
we may assume that $g = \tau_s$ for some non-torsion point $s \in E(K)$, and  
we obtain 
\begin{equation*}
\tau_{\varphi(h^m)\circ s} = \tau_{[-1]^m s}    
\end{equation*}
and thus 
\begin{equation}\label{eq:[-1]^ms}
\varphi(h^m)(s) = [-1]^m s.     
\end{equation}
We now consider the endomorphism $\psi \coloneq \varphi(h^m)-[-1]^m$ of $E$.  
Then \eqref{eq:[-1]^ms} means that $s\in \ker \psi$. 
Since $s$ is a non-torsion point, and any nonzero endomorphism of an elliptic curve 
has finite kernel (see \cite[Corollary III.4.9]{Silverman_AEC}), 
it follows that $\psi = 0$, that is, $\varphi(h^m) = [-1]^m$. 
If $m$ is even, then $\varphi(h^m) = \id_E$, which contradicts $\ord(\varphi(h^m)) > 1$. 
Therefore, $m$ is odd, and we obtain $\varphi(h^m) = [-1]$. 

The final assertion follows by the same argument as in \cref{prop:exsistenceof-1}. 
The proof is complete. 
\end{proof}

\subsection{Examples}\label{ss:Examples}
We give three examples of complex projective K3 surfaces $X$ with an elliptic fibration $p:X\to B$
of positive Mordell--Weil rank. 
In the first example, every parabolic automorphism in $\Aut_p(X)$ is chiral. 
In the second example, $\Aut_p(X)$ is uniformly achiral, but $\varphi(\Aut_p(X))$ does not 
contain $[-1]$. In the third example, $\Aut_p(X)$ contains both achiral parabolic automorphisms 
and chiral parabolic automorphisms.  
To show the chirality of parabolic automorphisms, the following lemma is useful. 

\begin{lem}\label{lem:h-lies-in-Aut_p}
Let $X$ be a K3 surface over an algebraically closed field, and 
let $p:X\to B$ be an elliptic fibration. 
Let $h\in \Aut(X)$, and suppose that 
$h\circ g^m \circ  h^{-1} = g^{-m}$ for some parabolic automorphism $g\in \Aut_p(X)$ 
and some integer $m\geq 1$. Then $h \in \Aut_p(X)$. 
\end{lem}
\begin{proof}
It is known that every parabolic isometry of the hyperbolic space $\HH_X$ has 
a unique \emph{limit point}; see \cite[\S12.1]{Rat}. 
In our case, the limit point of $g_*$ is corresponding to the fiber class $f\in \Num(X)$ of $p$. 
Since $g_*^{-m} = h_*g_*^mh_*^{-1}$ and $g_* f= f$, we have 
\[ g_*^{-m} (h_*f) = (h_*g_*^mh_*^{-1})(h_*f) = h_*g_*^mf = h_*f. \]
This means that $h_*f = f$ since $f$ is the unique limit point of $g_*^{-m}$. 
Hence $h\in \Aut_p(X)$. 
\end{proof}

Our examples are based on lattice theory. 
We begin by recalling some results and terminology from this theory. 
Let $L = (L, (\cdot, \cdot))$ be a lattice. We define 
$L^\vee \coloneq \{ y\in L\otimes\QQ \mid (y,L)\subset \ZZ \}$. 
The quotient $L^\vee/L$ is called the \emph{discriminant group} of $L$ and is denoted by $D_L$. 
Note that any isometry of $L$ naturally induces an automorphism of $D_L$. 

Let $L$ be an even lattice. 
For two elements $f, v\in L$ satisfying $f^2 = 0$ and $(f,v) = 0$, we define 
an isometry $t_{f,v} \in \mathrm{O}(L)$ by
\[t_{f,v}(x) = x + (x, v)f - \frac{1}{2}v^2(x,f)f - (x,f)v   \quad (x\in L). \]
This isometry is called the \emph{elementary transformation} associated with $f$ and $v$. 
We will need the following properties of elementary transformations, 
which follow from straightforward computations. 

\begin{lem}\label{lem:properties_of_elem_transf} 
Let $L$ be an even lattice, and let $f\in L$ be an element with $f^2 = 0$. 
\begin{enumerate} 
\item $t_{f,v}(f) = f$ for any $v\in L$. 
\item The action of $t_{f,v}$ on the discriminant group $D_L$ is trivial for any $v\in L$. 
\item $t_{f, 0} = \id$ and $t_{f,v} \circ t_{f,v'} = t_{f,v+v'}$ for any $v, v'\in L$. 
\item $\psi\circ t_{f, v}\circ \psi^{-1} = t_{\psi(f),\psi(v)}$
for any $v\in L$ and any $\psi\in \mathrm{O}(L)$. \qed
\end{enumerate}
\end{lem}

Let $L$ be an even lattice of signature $(1, \rho-1)$, where $\rho \coloneq \rk L$. 
Let $\mathrm{O}^+(L)$ denote the subgroup of $\mathrm{O}(L)$ consisting of isometries that preserve 
each of the two connected components of $\{ x\in L\otimes\RR \mid x^2 > 0 \}$. 
Then, every elementary transformation $t_{f, v}$ belongs to $\mathrm{O}^+(L)$, since 
$f$ lies on the boundary of one of the two components and $t_{f,v}(f) = f$.

We need the following results which establish the connection between lattices and K3 surfaces. 

\begin{prop}\label{prop:givenNS}
Let $\rho$ be a positive integer with $\rho \leq 11$, and let $L$ be an even lattice of 
signature $(1, \rho -1)$. Then, there exists a complex projective K3 surface $X$
such that $\NS(X) \cong L$.   
\end{prop}
\begin{proof}
This follows from Nikulin's theorem on primitive embeddings and from 
the surjectivity of the period mapping; see \cite[Remark 2.11]{Morrison1984}. 
\end{proof}

\begin{prop}\label{prop:Torelli}
Let $X$ be a complex projective K3 surface, and let $t$ be an isometry of $\NS(X)$. 
Suppose that $t$ preserves the nef cone and acts trivially on $D_{\NS(X)}$. 
Then, there exists a unique symplectic automorphism $g$ of $X$ such that 
$g_*|_{\NS(X)} = t$. 
\end{prop}
\begin{proof}
This is a consequence of the Torelli-type theorem 
(see e.g. \cite[Theorem 6.1]{Kondo_K3book}) and \cite[Lemma 8.11]{Kondo_K3book}.  
\end{proof}

For a lattice $L = (L, (\cdot, \cdot))$ and a nonzero integer $k$, we define 
$L(k)\coloneq (L, k(\cdot, \cdot))$. 
We write $A_n$, $E_n$, and $U$ for the root lattice of $A_n$-type, that of $E_n$-type, 
and the hyperbolic plane lattice, respectively.  

Let $L'$ be a negative definite lattice of rank at most $9$ that 
can be written as $L' = L''(2)$ for some even lattice $L''$. 
Then, by \cref{prop:givenNS}, there exists a K3 surface $X$ such that $\NS(X) = U(4)\oplus L'$. 
Furthermore, the closure of the positive cone coincides with the nef cone, since there 
is no element in $\NS(X)$ with self-inner product $-2$. 
Let $F\in \NS(X)$ be a primitive element in $U(4)$ with $F^2 = 0$. 
After replacing $F$ by $-F$ if necessary, we may assume that $F$ is nef. 
Then, we obtain an elliptic fibration $p:X\to \PP^1$
as the morphism associated with the complete linear system $|F|$; 
see \cite[Proposition 2.3.10]{Huybrechts_K3book}.  

\begin{lem}\label{lem:for_examples}
Let $X$ be a complex projective K3 surface with $\NS(X) = U(4)\oplus L'$ as above. 
\begin{enumerate}
\item For any $v\in L'$, there exists a unique symplectic automorphism $g_v\in \Aut_p(X)$ 
such that $g_{v,*}|_{\NS(X)} = t_{F, v}$. 
\item The map $\theta: L'\to \Aut_p(X),\; v \mapsto g_v$ is an injective group homomorphism.  
\item For any parabolic automorphism $g\in \Aut_p(X)$, there exist $m,n\in \ZZ_{>0}$ and 
$v\in L'$ such that $g^m = g_v^n$. 
\item For any $v\in L'$ and $h\in \Aut_p(X)$, we have  
$h\circ g_v \circ h^{-1} = g_{h_*v}$. 
\item For any $v\in L'$, $g_v$ is achiral if and only if there exists 
$h \in \Aut_p(X)$ such that $h_* v = - v$. 
\end{enumerate}
\end{lem}
\begin{proof}
(i). Let $v \in L'$. Then, the action of $t_{F,v}$ on $D_{\NS(X)}$ is the identity by 
\cref{lem:properties_of_elem_transf} (ii). Moreover, $\phi_{F,v}$ preserves the nef cone, 
which now coincides with the positive cone, since $t_{F,v}$ belongs to $\mathrm{O}^+(\NS(X))$. 
Hence, the statement follows from \cref{prop:Torelli}.   

(ii). Since $v\mapsto t_{F,v}$ is a group homomorphism by \cref{lem:properties_of_elem_transf} (iii), 
it follows that $\theta$ is also a group homomorphism. 
Let $v\in L'$ be a nonzero element. Note that $v^2 \neq 0$ since $L'$ is negative definite. 
Then $t_{F,v}(v) = v+ v^2 F\neq v$, which shows that $t_{F,v}\neq \id$. 
Hence, $g_v \neq \id$, and $\theta$ is injective. 

(iii). Since $\Aut_p(X)$ contains $\MW(p)$ with finite index, 
it suffices to show that $\MW(p)\otimes\QQ = \theta(L')\otimes \QQ$, 
or equivalently, that $\MW(p)$ and $\theta(L')$ have the same rank (as abelian groups). This follows from the Shioda--Tate formula. 

(iv). Let $v\in L'$ and $h\in \Aut_p(X)$. We have 
\[ (h \circ g_v \circ h^{-1})_*|_{\NS(X)} 
= (h_*|_{\NS(X)})\circ t_{F, v} \circ h^{-1}_*|_{\NS(X)}
= t_{F, h_*v}
\]
by \cref{lem:properties_of_elem_transf} (iv). 
Therefore, $h \circ g_v \circ h^{-1} = g_{h_*v}$. 

(v). Let $v\in L'$. For any $m\in \ZZ$ and $h\in \Aut_p(X)$, we have 
$g_v^{-m} = g_{-mv}$ by (ii), and 
$h g_v^m h^{-1} = h g_{mv} h^{-1} = g_{h_*(mv)}$ by (iv). 
Hence, we have $h g_v^m h^{-1} = g_v^{-m}$ if and only if $h_*v = -v$. 
Together with \cref{lem:h-lies-in-Aut_p}, we obtain the statement. 
\end{proof}

We make use of the following lemma to observe examples.   

\begin{lem}\label{lem:odd_picard_number}
Let $X$ be a complex K3 surface, and let $L \coloneq \NS(X)$. 
\begin{enumerate}
\item If the Picard number $\rho(X) = \rk L$ is odd, then 
any automorphism $h$ of $X$ is either symplectic or anti-symplectic. 
\item If $h$ is a symplectic (resp. anti-symplectic) automorphism of $X$, 
then $\overline{h_*|_L} = \id_{D_L}$ (resp. $\overline{h_*|_L} = -\id_{D_L}$), 
where $\overline{h_*|_L}$ is the automorphism of $D_L$ induced by the isometry $h_*|_L$. 
\end{enumerate}
\end{lem}
\begin{proof}
This follows from \cite[Corollary 8.13]{Kondo_K3book}. 
\end{proof}

\begin{eg}\label{ex:chiral}
Let $X$ be a K3 surface such that $\NS(X) = U(4)\oplus A_1(-2)$. 
We see that every parabolic element in $\Aut_p(X)$ is chiral. 
To this end, by \cref{lem:for_examples}, it suffices to show that   
there exists no automorphism $h\in \Aut_p(X)$ such that $h_*u = -u$, 
where $u\in \NS(X)$ is a generator of $A_1(-2)$. 
Suppose that such an $h$ existed. 
Note that $\frac{1}{4}F$ and $\frac{1}{4}u$ lie in $\NS(X)^\vee$. 
We have $h_*(\frac{1}{4}F) = \frac{1}{4}F$ and $h_*(\frac{1}{4}u) = -\frac{1}{4}u$. 
This would imply that the induced action on $D_{\NS(X)}$ of $h_*$ is neither $\id_{D_{\NS(X)}}$ 
nor $-\id_{D_{\NS(X)}}$, which contradicts \cref{lem:odd_picard_number}. 
Hence, every parabolic element in $\Aut_p(X)$ is chiral. 
\end{eg}

\begin{eg}\label{ex:uniformly_achiral}
Let $X$ be a K3 surface such that $\NS(X) = U(4)\oplus E_8(-2)$. 
Then $\varphi(\Aut_p(X))$ does not contain $[-1]$, since the multisection index is $4$. 
Nevertheless, we see that $\Aut_p(X)$ is uniformly achiral. 
Let $r = \id_{U(4)} \oplus -\id_{E_8(-2)}\in \mathrm{O}(\NS(X))$. Then 
$r(v) = -v$ for any $v\in E_8(-2)$. On the other hand,  
we have $-\id_{D_{E_8(-2)}} = \id_{D_{E_8(-2)}}$ since $D_{E_8(-2)} \cong (\ZZ/2\ZZ)^8$, 
and thus $r$ acts trivially on $D_{\NS(X)} = D_{U(4)} \oplus D_{E_8(-2)}$. 
Hence, by \cref{prop:Torelli}, there exists an automorphism $h\in \Aut_p(X)$ such that 
$h_*|_{\NS(X)} = r$,  
and \cref{lem:for_examples} shows that this $h$ makes $\Aut_p(X)$ uniformly achiral. 
\end{eg}

\begin{eg}
Let $X$ be a K3 surface such that $\NS(X) = U(4)\oplus A_1(-2) \oplus E_8(-2)$. 
We see that $\Aut_p(X)$ contains both achiral parabolic automorphisms and chiral parabolic 
automorphisms. More precisely, $g_v$ is achiral for $v \in E_8(-2)$, and 
$g_u$ is chiral for $u \in A_1(-2)$. 

Let $r = \id_{U(4)\oplus A_1(-2)} \oplus -\id_{E_8(-2)}\in \mathrm{O}(\NS(X))$. 
Then, as in \cref{ex:uniformly_achiral}, there exists an automorphism $h\in \Aut_p(X)$ 
such that $h_*|_{\NS(X)} = r$, and $h$ makes $g_v$ achiral for any $v\in E_8(-2)$. 
To show that $g_u$ is chiral for $u \in A_1(-2)$, it is enough to consider the case 
where $u$ is a generator of $ A_1(-2)$, and the claim follows from the same argument 
as in \cref{ex:chiral}.  
\end{eg}

\bibliography{math}
\bibliographystyle{alpha}
\end{document}